\documentclass[11pt]{amsart}
\usepackage[margin=3cm]{geometry}

\usepackage[utf8]{inputenc}
\usepackage{mathtools,amsfonts,amsthm,mathrsfs,amssymb}
    \usepackage{stmaryrd}
\usepackage{textcomp}
\mathtoolsset{showonlyrefs}
\usepackage{bookmark}
\usepackage{microtype}
\usepackage{todonotes}
\usepackage{enumerate}

\newtheorem{thm}{Theorem}[section]
\newtheorem{lemma}[thm]{Lemma}

\newtheorem{corollary}[thm]{Corollary}

\newtheorem{obs}[thm]{Observation}
\theoremstyle{definition}
\newtheorem*{definition*}{Definition}

\newtheorem{qn}{Question}
\newtheorem*{claim}{Claim}



\newtheorem*{acknowledgement}{Acknowledgments}

\newcommand{\N}{\mathbb{N}}

\newcommand{\diam}{\mathrm{diam}}

\renewcommand{\le}{\leqslant}
\renewcommand{\leq}{\leqslant}
\renewcommand{\ge}{\geqslant}
\renewcommand{\geq}{\geqslant}

\title{Surfaces have (asymptotic) dimension 2}

\author[M. Bonamy]{Marthe Bonamy} \address{LaBRI, CNRS,
  Universit\'e de Bordeaux, Bordeaux, France}
\email{marthe.bonamy@u-bordeaux.fr}

\author[N. Bousquet]{Nicolas Bousquet}
\address{LIRIS, CNRS, Universit\'e Claude Bernard Lyon 1, Lyon, France.}
\email{nicolas.bousquet@univ-lyon1.fr}

\author[L. Esperet]{Louis Esperet}
\address{Laboratoire G-SCOP, CNRS, Univ. Grenoble Alpes, Grenoble, France.}
\email{louis.esperet@grenoble-inp.fr}

\author[C. Groenland]{Carla Groenland}
\address{Mathematical  Institute,  University  of  Oxford,  Oxford  OX2  6GG,  United  Kingdom.}
\email{carla.groenland@maths.ox.ac.uk}

\author[F.\ Pirot]{Fran\c{c}ois Pirot}
\address{Laboratoire G-SCOP, CNRS, Univ. Grenoble Alpes, Grenoble, France.}
\email{francois.pirot@grenoble-inp.fr}

\author[A. Scott]{Alex Scott}
\address{Mathematical  Institute,  University  of  Oxford,  Oxford  OX2  6GG,  United  Kingdom.}
\email{scott@maths.ox.ac.uk}

\thanks{M. \ Bonamy and N.\ Bousquet are supported by ANR Projects DISTANCIA (\textsc{ANR-17-CE40-0015}) and GrR (\textsc{ANR-18-CE40-0032}). 
L.\ Esperet and F. Pirot are supported by ANR Projects GATO
(\textsc{ANR-16-CE40-0009-01}) and GrR (\textsc{ANR-18-CE40-0032}). }

\begin{document}
\begin{abstract}
The asymptotic dimension is an invariant of metric spaces introduced by Gromov in the context of geometric group theory. When restricted to graphs and their shortest paths metric, the asymptotic dimension can be seen as a large scale version of weak diameter colourings (also known as weak diameter network decompositions), i.e.\ colourings in which each monochromatic component has small weak diameter. 

In this paper, we prove that for any $p$, the class of graphs
excluding $K_{3,p}$ as a minor has asymptotic dimension at most 2. This
implies that the class of all graphs embeddable on any fixed surface
(and in particular the class of planar graphs) has asymptotic
dimension 2, which gives a positive answer to a recent question of
Fujiwara and Papasoglu. Our result extends from graphs to Riemannian
surfaces. We also prove that graphs of bounded pathwidth have
asymptotic dimension at most 1 and graphs of bounded layered pathwidth have
asymptotic dimension at most 2. We give some applications of our techniques to
graph classes defined in a topological or geometrical way, and to
graph classes of polynomial growth. Finally we prove that the class of
bounded degree graphs from any fixed proper minor-closed class has
asymptotic dimension at most 2. This can be seen as a large scale
generalization of the result that bounded degree graphs from any fixed
proper minor-closed class are 3-colourable with monochromatic
components of bounded size. This also implies that (infinite) Cayley graphs
avoiding some minor have asymptotic dimension at most 2, which solves a
problem raised by Ostrovskii and Rosenthal.
\end{abstract}
\maketitle

\section{Introduction}\label{sec:intro}

\subsection{Asymptotic dimension}

Let $(X,d)$ be a metric space, and let $\mathcal{U}$ be a family of subsets of
$X$. We say that $\mathcal{U}$ is \emph{$D$-bounded} if each set $U\in \mathcal{U}$ has diameter at
most $D$. We say that $\mathcal{U}$ is \emph{$r$-disjoint} if for any
$a,b$ belonging to different elements of $\mathcal{U}$ we have
$d(a,b)> r$.

We say that $D_X:\mathbb{R}^+\to
\mathbb{R}^+$ is an
\emph{$n$-dimensional control function} for $X$ if for any $r>0$, $X$ has a cover
$\mathcal{U}=\bigcup_{i=1}^{n+1}\mathcal{U}_i$, such that each
$\mathcal{U}_i$ is $r$-disjoint and each element of $\mathcal{U}$ is
$D_X(r)$-bounded. The \emph{asymptotic dimension} of $X$, denoted by $\mathrm{asdim}\, X$, is the
least integer $n$ such that $X$ has an $n$-dimensional control function. If no such integer $n$ exists, then the asymptotic
dimension is infinite. 
This notion was introduced by Gromov~\cite{Gr93} in the
context of geometric group theory. The reader is referred to~\cite{BD08} for a
survey on asymptotic dimension and its group theoretic applications,
and to the lecture notes of Roe~\cite{Roe03} on coarse geometry for more detailed proofs of some results of~\cite{Gr93}.

As the asymptotic dimension of a bounded space is 0, in the context of
finite metric spaces we are more interested in the asymptotic dimension of (infinite)
\emph{classes} of metric spaces. We define the \emph{asymptotic dimension of a
family} $\mathcal{X}$ of metric spaces as the least $n$ for which
there exists a function $D_{\mathcal{X}}:\mathbb{R}^+\to \mathbb{R}^+$
which is an $n$-dimensional control function for each $X\in \mathcal{X}$.

\subsection{Graphs as metric spaces}

Given a graph $G$ and a collection of positive reals
$\ell=(\ell_e)_{e\in E(G)}$, the \emph{weighted graph} $(G,\ell)$ is
the discrete metric space whose points are the vertices of $G$, and whose metric coincides with the weighted shortest path metric in
$G$ (where each edge $e$ is considered with its weight $\ell_e$). The
distance between two vertices $u$ and $v$ in this metric is denoted by
$d_{(G,\ell)}(u,v)$ (or simply $d_G(u,v)$ if the weights are clear
from the context). Note that if all weights are equal to 1,  $d_G(u,v)$ is the number of edges of a shortest path between $u$ and
$v$ in $G$. If all the weights of a weighted graph $G$ lie in some real interval $I$, we say that $G$ is an \emph{$I$-weighted graph}.
It will be convenient at points to restrict ourselves to $(0,1]$-weighted graphs, but as the following observation shows, we will not lose generality by doing so (whenever we consider weighted graphs in the remainder, the classes of graphs under consideration are closed under taking subdivisions).

\begin{obs}\label{obs:subv}
 Let $(G,\ell)$ be a weighted graph and let $(G',\ell')$ be obtained from $(G,\ell)$ by subdividing some egde $e$ of $G$ once (i.e.\ replacing $e$ by a path of length two), and assigning weights summing to $\ell_e$ to  the two newly created edges, while all the other edges of $G$ retain their weight from $\ell$. Then any $n$-dimensional control function for $(G',\ell')$ is also an $n$-dimensional control function for $(G,\ell)$. 
\end{obs}

What can be said about the asymptotic dimension of graphs (weighted
or not)? As observed above, a finite graph has asymptotic dimension
0, so this question is only interesting for infinite graphs, or for
infinite classes of (finite or infinite) graphs.

\smallskip

The class of trees has asymptotic dimension 1~\cite{BD08}. Gromov~\cite{Gr93} observed 
that $d$-dimensional Euclidean spaces have asymptotic
dimension $d$, and it can easily be deduced from this that for any $d\ge 1$, the class
of $d$-dimensional grids (with or without diagonals) has asymptotic dimension $d$. 
On the other hand, it was proved that any infinite family of bounded degree
expanders (in particular cubic expanders) has unbounded asymptotic
dimension~\cite{Hum17}. Another example of a class of graphs with bounded degree
and infinite asymptotic dimension is the class of lamplighter
graphs of binary trees~\cite{BMSZ20} (these graphs have maximum degree 4). This implies that bounding the degree is not
enough to bound the asymptotic dimension.

\subsection{Graphs on surfaces}
It was proved
in~\cite{OR15} that the class of graphs excluding the complete graph $K_t$ as a
minor has asymptotic dimension at most $4^t$. Recently, Fujiwara and
Papasoglu~\cite{FP20} proved that the class of planar graphs has asymptotic
dimension at most 3, and asked whether this can be improved to 2.
Combining results from their paper with a quantitative version of a theorem of 
Brodskiy, Dydak, Levin and Mitra~\cite{BDLM}, we begin by showing that this is indeed the case.

\begin{thm}\label{thm:introplanar}
Planar graphs
have asymptotic dimension at most 2.
\end{thm}

In fact we will prove a stronger result (Theorem \ref{thm:planar}), showing that we can choose a 2-dimensional control function of form $f(r)=O(r)$.
Note that the fact that the asymptotic dimension is at least 2 for this class follows from the example of 2-dimensional grids. 

Fujiwara and
Papasoglu~\cite{FP20} showed that their result on planar graphs
extends to any geodesic Riemannian plane (i.e. any
geodesic Riemannian surface homeomorphic to $\mathbb{R}^2$). We will prove
that our result also extends to this setting (Theorem \ref{thm:planar}).  We note that 
shortly after we posted our manuscript on the arXiv repository, we were informed by J\o rgensen and Lang that they had independently proved a version of the results above for planar graphs and geodesic Riemannian planes in~\cite{JL20}, using essentially the same tools.

What happens on surfaces of higher genus?  Our central result here is the following.

\begin{thm}\label{thm:main}
 For any fixed $p\ge 3$, the class of graphs excluding the complete
bipartite graph $K_{3,p}$ as a minor has asymptotic dimension at most
2.
\end{thm}

Note that this result for $p=3$ already implies Theorem \ref{thm:introplanar}, since
since planar graphs exclude $K_{3,3}$ as a minor. 
More generally, for any fixed
integer $g\ge 0$, the class of graphs embeddable on a surface of Euler genus
$g$ excludes $K_{3,p}$ as a minor, for some fixed integer $p$, so our
result immediately implies the following result.
\begin{corollary}
For any $g\ge 0$, the class of graphs
embeddable on a surface of Euler genus $g$ has asymptotic dimension 2. 
\end{corollary}

\smallskip
 
We further prove
that our result extends to the setting of compact Riemannian
surfaces of Euler genus $g$, for any fixed $g\ge
0$. 

\begin{thm}
For any $g\ge 0$, the class of compact Riemannian surfaces of Euler genus $g$ has asymptotic dimension 2. 
\end{thm}

\subsection{Minor excluded graph classes and Cayley graphs}\label{sec:cayley}

A graph $H$ is a \emph{minor} of a graph $G$ if it can be obtained from $G$ by contracting some edges, and deleting some vertices and edges. We say that a class of graphs $\mathcal{G}$ is \emph{minor-closed} if any minor of a graph from $\mathcal{G}$ is also in $\mathcal{G}$. A minor-closed class $\mathcal{G}$ is \emph{proper} if it does not contain all graphs (equivalently, if there is a graph $H$, such that no graph of $\mathcal{G}$ contains $H$ as a minor). Any minor of a graph embeddable on a surface $\Sigma$ is also embeddable in $\Sigma$, and thus classes of graphs embeddable on a fixed surface form a natural example of a proper minor-closed class. A deep theorem of Robertson and Seymour~\cite{RS04} shows that any proper minor-closed class can be characterised by a finite number of forbidden minors.

\smallskip

Fujiwara and
Papasoglu~\cite{FP20} raised the following problem.

\begin{qn}[Question 5.2 in~\cite{FP20}]\label{qn:fp}
Is there a constant $k$ such that for any graph $H$, the class of
graphs excluding $H$ as a minor has asymptotic dimension at most $k$? Can we
take $k=2$?
\end{qn}

Theorem~\ref{thm:main} proves that Question~\ref{qn:fp} has a
positive answer when
$H=K_{3,p}$, for any constant $p$, and we will prove in
Corollary~\ref{cor:apexforest} that the result also holds when $H$ is an
apex-forest (i.e.\ when $H$ contains a vertex $v$ such that $H-v$ is a
forest). This will be deduced from the result that any class of graphs of bounded pathwidth has asymptotic dimension 1  (Theorem~\ref{thm:pw}). In a different direction, we can prove a positive answer to Question~\ref{qn:fp} when $H$ is arbitrary but we restrict ourselves to $H$-minor free graphs of
bounded degree.

\begin{thm}\label{thm:minordeg}
For any integer $\Delta$ and graph $H$, the class of $H$-minor free
graphs of maximum degree at most $\Delta$
has asymptotic dimension at most 2.
\end{thm}

Note that the proof of Theorem~\ref{thm:minordeg} uses a recent result relying on the graph minor structure theorem by Robertson and Seymour~\cite{RS03}. 

\medskip

Given a finitely generated group $G$ and a finite generating set $S$
(assumed to be symmetric, in the sense that $s\in S$ if and only if
$s^{-1}\in S$), the \emph{Cayley graph} $\text{Cay}(G,S)$ is the graph with
vertex set $G$, with an edge between two elements $u,v\in G$ if and
only if $u=vs$ for some $s\in S$. 
As observed by Gromov~\cite{Gr93}, the
asymptotic dimension of $\text{Cay}(G,S)$ is independent of the choice
of the finite generating set $S$, and thus the asymptotic
dimension is a group invariant (this was the main motivation for introducing this invariant).

We say that $\text{Cay}(G,S)$ is
\emph{minor excluded} if it excludes some finite graph $H$ as a
minor. The following special case of Question~\ref{qn:fp} was raised by Ostrovskii and Rosenthal~\cite{OR15}.

\begin{qn}[Problem 4.1 in~\cite{OR15}]\label{qn:or}
  Let $G$ be a finitely generated group and $S$ a finite generating
  set such that $\text{Cay}(G,S)$ is minor excluded. Does it follow
  that $G$ has  asymptotic dimension at most 2?
\end{qn}

A simple compactness argument (see for instance~\cite{Got51}) shows that the asymptotic dimension of
an infinite unweighted graph is at most the asymptotic dimension of the class
of its finite induced subgraphs. If $\text{Cay}(G,S)$ excludes some
minor $H$, then since $S$ is finite, all induced subgraphs of
$\text{Cay}(G,S)$ have degree at most $|S|$ and exclude $H$ as a
minor. We thus obtain the following positive answer to
Question~\ref{qn:or} as an immediate consequence of
Theorem~\ref{thm:minordeg}. 

\begin{corollary}
Let $G$ be a finitely generated group and $S$ a finite generating
  set such that $\text{Cay}(G,S)$ is minor excluded. Then $G$ has  asymptotic dimension at most 2.
\end{corollary}

\subsection{Polynomial growth}

For some function $f$, a graph $G$ has \emph{growth} at most $f$ if for any integer $r$, any vertex $v\in V(G)$ has at most $f(r)$ vertices at distance at most $r$.  
It is known that vertex-transitive graphs of polynomial growth have bounded asymptotic dimension, while some classes of graphs of exponential growth have unbounded asymptotic dimension~\cite{Hum17}. 

We first prove the following result.

\begin{thm}\label{thm:subgrid}
For any $d\ge 1$, the class of (not necessarily induced) subgraphs of the $d$-dimensional grid has asymptotic dimension at most $d$.
\end{thm}

Using a result of Krauthgamer and Lee~\cite{KL03}, we then deduce the following.

\begin{corollary}
  Any class of graphs of polynomial growth has bounded asymptotic dimension
\end{corollary}

We also prove that this is sharp: for any superpolynomial function $f$ we construct a class of graphs of growth at most $f$ with unbounded asymptotic dimension.

\subsection{Weak diameter colouring and clustered colouring}\label{sec:clustered}

The \emph{weak diameter} of a subset $S$ of vertices of a graph
$G$ is the maximum distance (in $G$) between two vertices of $S$.
A graph $G$ is \emph{$k$-colourable with weak diameter $d$} if each vertex
of $G$
can be assigned a colour from $\{1,\ldots,k\}$ in such a way that all
monochromatic components (i.e. connected components of the subgraph
induced by a colour class) have weak diameter at most $d$. This notion
is also studied under the name of \emph{weak diameter network
  decomposition} in distributed computing (see~\cite{AGLP}), although in this context $k$ and $d$
usually depend on $|V(G)|$ (they are typically of order $\log|V(G)|$),
while here we will only consider the case where $k$ and $d$ are constant.
Observe that the case $d=0$ corresponds to the usual notion of (proper)
colouring. Note also that this property should not be confused with the stronger property that the
subgraph induced by each monochromatic component has bounded diameter
(see for instance~\cite[Theorem 4.1]{LO18}).   We say that a class of graphs $\mathcal{G}$ has \emph{weak
  diameter chromatic
  number} at most $k$ if there is a constant $d$ such that every graph
of $\mathcal{G}$ is $k$-colourable with weak diameter $d$. 

\begin{obs}\label{obs:weakdiameter}
If a class of graphs $\mathcal{G}$ has 
asymptotic dimension at most $k$, then $\mathcal{G}$ has weak diameter
chromatic number at most $k+1$. 
\end{obs}

\begin{proof}
Taking $r=1$ in the definition of asymptotic dimension, we obtain that
there is a constant $D$ such that any graph $G\in\mathcal{G}$ has a cover by $k+1$ 1-disjoint  families of
$D$-bounded sets. This cover can be reduced to a partition of the
vertex set of $G$ (while maintaining the property that the $k+1$
families are 1-disjoint and all their elements are $D$-bounded). We can now consider each of the $k+1$ families
as a distinct colour class. Since each family is 1-disjoint, each monochromatic
component is included in a single element of the partition, and is
therefore $D$-bounded. This implies that $\mathcal{G}$ has weak diameter chromatic number at most
$k+1$. 
\end{proof}

As a consequence, Theorem~\ref{thm:main} has the following immediate corollary.

\begin{corollary}\label{cor:wdcol}
  For any fixed $p\ge 3$, the class of graphs with
no $K_{3,p}$-minor (and in particular planar graphs and any class of graphs embeddable
on a fixed surface) has weak diameter chromatic number at most 3.
\end{corollary}

A graph $G$ is \emph{$k$-colourable with clustering $c$} if each vertex
of $G$
can be assigned a colour from $\{1,\ldots,k\}$ in such a way that all
monochromatic components have size at most $c$. The case $c=1$
corresponds to the usual notion of (proper) colouring, and there is a
large body of work on the case where $c$ is a fixed constant (see~\cite{Woo18} for a recent survey). 
In this context, we
say that a class of graphs $\mathcal{G}$ has \emph{clustered chromatic
  number} at most $k$ if there is a constant $c$ such that every graph
of $\mathcal{G}$ is $k$-colourable with clustering $c$.

\begin{obs}\label{obs:cluster}
If a class of graphs $\mathcal{G}$ has maximum degree at most $\Delta$ and
asymptotic dimension at most $k$, then $\mathcal{G}$ has clustered
chromatic number at most $k+1$. 
\end{obs}

\begin{proof}
By Observation~\ref{obs:weakdiameter}, there is a constant $d$ such
that any graph $G\in \mathcal{G}$ has a $(k+1)$-colouring with weak
diameter at most $d$. Consider a graph $G\in \mathcal{G}$ and such a
colouring. Since $G$
has maximum degree $\Delta$, each monochromatic component has size at
most $\Delta^d$. This implies that $\mathcal{G}$ has clustered chromatic number at most
$k+1$. 
\end{proof}

It is known
that for every graph $H$ and integer $\Delta$, the class of $H$-minor
free graphs of maximum degree $\Delta$ has clustered chromatic number
at most 3~\cite{LO18,DEMWW}. This is now also a direct consequence of 
Observation~\ref{obs:cluster} and Theorem~\ref{thm:minordeg}, and in fact Theorem~\ref{thm:minordeg} can be seen as a
large scale generalization of results on clustered colouring.

\subsection{Linear type and Assouad-Nagata dimension}

Gromov~\cite{Gr93} noticed that the notion of asymptotic dimension of
a metric space can
be refined by restricting the growth rate of the control function $D(r)$ in
its definition. Although this function can be chosen to be linear in
many cases, its complexity can be significantly worse in general, as
there exist (Cayley) graphs of asymptotic dimension $n$ for which any
$n$-dimensional control function $D(r)$ grows as fast as $\Omega(\exp \exp \cdots \exp r^k)$,
for any given height of the tower of exponentials~\cite{Now07}.
A control function $D_X$ for a metric space $X$ is said to be
\emph{linear} if there is a constant $c>0$ such that $D_X(r)\le
cr+c$ for any $r>0$.
We say that a metric space $(X,d)$ has \emph{asymptotic dimension at
  most $n$ of linear type}  if $X$ has a linear $n$-dimensional
control function. The definition extends
to families of metric spaces in a natural way. This notion is sometimes called \emph{asymptotic dimension with
Higson property}~\cite{DZ04}.  Nowak~\cite{Now07} proved that the asymptotic dimension of linear type is not bounded by any function of the asymptotic dimension, by constructing (Cayley) graphs of asymptotic dimension 2 and infinite asymptotic dimension of linear type.
It turns out that some of our main
results hold for the
asymptotic dimension of linear type, i.e.\ the diameter bound that we
obtain is linear in the disjointness parameter $r$.

\medskip

Asymptotic dimension of linear type is related to the following
notion.
A control function $D_X$ for a metric space $X$ is said to be
\emph{a dilation} if there is a constant $c>0$ such that $D_X(r)\le
cr$, for any $r>0$.
We say that a metric space $(X,d)$ has \emph{Assouad-Nagata dimension at
  most $n$}  if $X$ has an $n$-dimensional control function which is a dilation. The Assouad-Nagata dimension was introduced by Assouad~\cite{As82}
(see~\cite{LS05} for more results on this notion). The
main difference between this notion and the asymptotic dimension of
linear type is that the latter is a large-scale dimension (we can
assume that $r$ is arbitrarily large by tuning the constant $c$ if
necessary), while the former is a dimension at all scales: it sheds
some light on the geometry of the space as $r\to \infty$
(large-scale), but also as $r\to 0$ (microscopic scale). Note that the
two notions are equivalent for
\emph{uniformly discrete metric spaces}, i.e.\ metric spaces such that
all pairs of distinct points are at distance at least $\epsilon>0$ apart, for
some universal $\epsilon>0$. This is the
case for graphs (two distinct vertices are at distance at least 1
apart), and for weighted graphs with a uniform lower bound on the edge
weights. It follows that in these cases, some of our main results hold for
the Assouad-Nagata dimension as well.

\subsection{Sparse partitions}\label{sec:sparsecover}

A \emph{ball of radius $r$} (or \emph{$r$-ball}) centered in a
point $x$, denoted by $B_r(x)$, is the set of points of $X$ at distance at
most $r$ from $x$.
For a real $r\ge 0$ and an integer $n\ge 0$, the family $\mathcal{U}$ has \emph{$r$-multiplicity}
at most $n$ if each $r$-ball in $X$ intersects at most $n$ sets of
$\mathcal{U}$. It is not difficult to see that if $D_X(r)$ is an
$n$-dimensional control function for a metric space $X$, then for any
$r>0$, $X$ has a $D_X(2r)$-bounded cover of $r$-multiplicity at most
$n+1$.
Gromov~\cite{Gr93} proved that a converse of this result also holds,
in the sense that
the asymptotic dimension of $X$ is exactly the
least integer $n$ such that for any real number $r\ge 0$, there is a real number $D_X'(r)$
such that $X$ has a $D_X'(r)$-bounded cover of $r$-multiplicity at most
$n+1$. Moreover, the function $D_X'$ has the same type as the
$n$-dimensional control function of $D_X$ of $X$: $D_X'$ is linear if and
only if $D_X$ is linear, and $D_X'$ is a dilation if and only if
$D_X$ is a dilation.

As a consequence, the notions of Assouad-Nagata dimension and asymptotic dimension of
linear type are closely related to the well-studied notions of sparse covers and
sparse partitions in theoretical computer science. A weighted graph
$G$ admits a \emph{$(\sigma,\tau)$-weak sparse partition scheme} if for
any $r\ge 0$,  the vertex set of $G$ has a partition into $(\sigma\cdot
r)$-bounded sets of $r$-multiplicity at most $\tau$, and such a
  partition can be computed in polynomial time. As before, we say
that a family of graphs admits a $(\sigma,\tau)$-weak sparse partition
scheme if all graphs in the family admit a $(\sigma,\tau)$-weak sparse partition scheme. This definition was
introduced in~\cite{JLNRS05}, and is equivalent to
the notion of weak sparse cover scheme of Awerbuch and
Peleg~\cite{AP90} (see~\cite{Fil20}). Note that if a family of graphs
admits a $(\sigma,\tau)$-weak sparse partition scheme then its
Assouad-Nagata dimension is at most $\tau-1$.  
Conversely, if a family
of graphs has Assouad-Nagata dimension at most $d$ and the covers can
be computed efficiently, then the family admits a $(\sigma,d+1)$-weak
sparse partition scheme, for some constant $\sigma$.

While the two notions are almost equivalent, it should be noted that
the emphasis is on different parameters. In the case of the
Assouad-Nagata dimension, the goal is to minimize $d$ (or equivalently
$\tau$ in the sparse partition scheme), while in the
$(\sigma,\tau)$-weak sparse partition scheme, the goal is usually to
minimize a function of $\sigma$ and $\tau$ which depends on the
application. As an example, it was proved in~\cite{JLNRS05} that if an
$n$-vertex  graph admits a $(\sigma,\tau)$-weak sparse partition
scheme, then the graph has a \emph{universal steiner tree} with
stretch $O(\sigma^2 \tau \log_\tau n)$, so in this case the goal is to
minimize $\sigma^2\cdot \tfrac{\tau}{\log \tau}$.

All our proofs are constructive and give polynomial-time algorithms to
compute the covers. The following is a sample of our results (translated into the terminology of
sparse partition schemes).
\begin{itemize}
\item trees (and more generally chordal graphs) admit $(O(1),2)$-weak partition schemes (this actually follows directly from the results of~\cite{FP20});
\item graphs excluding $K_{3,p}$ as a minor admit 
  $(O(p^2),3)$-weak partition schemes, while using the original
  approach of~\cite{FP20}, it can be shown that they admit $(O(p),4)$-weak partition schemes;
  \item in particular, graphs of Euler genus $g\ge 0$ admit
    $(O(g^2),3)$-weak partition schemes and $(O(g),4)$-weak partition schemes.
\end{itemize}

\subsection{Outline of the paper}

In Section~\ref{sec:control}, we introduce some terminology and give a quantitative version of
a result of Brodskiy, Dydak, Levin
and Mitra~\cite{BDLM} (Theorem~\ref{thm:mthm49}, part of which is not explicitly stated in~\cite{BDLM} and is proved in the appendix for completeness). We then show how to deduce that planar graphs and Riemannian planes have asymptotic dimension 2 (Theorem~\ref{thm:planar}). After that, we use Theorem~\ref{thm:mthm49} again in the proof of our main tool (Theorem~\ref{thm:bd}), 
which states informally that if all graphs in a class $\mathcal{G}$ have a layering in which any constant number of layers induce a graph from a class of asymptotic dimension at most $k$, then $\mathcal{G}$ has asymptotic dimension at most $k+1$.

In Section~\ref{sec:k3p} we use Theorem~\ref{thm:bd} to prove that graphs excluding $K_{3,p}$ as a minor have asymptotic dimension 2 (Theorem~\ref{thm:k3p}). In Section~\ref{sec:surfaces} we show that this implies that any class of graphs embeddable on a fixed surface has asymptotic dimension at most 2 (Corollary~\ref{cor:genus2}). The fact that this result actually holds for weighted graphs is then used to deduce that for any integer $g\ge 0$, the class of (compact) Riemannian surfaces of Euler genus $g$ has asymptotic dimension 2.

In Section~\ref{sec:topo}, we prove that classes of graphs of bounded degree and bounded treewidth have asymptotic dimension 1 (Theorem~\ref{thm:twdeg}) and deduce that classes of graphs of bounded degree and bounded layered treewidth have asymptotic dimension 2 (Theorem~\ref{thm:ltwdeg}). In particular this implies Theorem~\ref{thm:minordeg}, stating that graphs of bounded degree from any proper minor-closed class have asymptotic dimension at most 2.

In Section~\ref{sec:geom} we explore some consequences of Theorem~\ref{thm:bd} for classes of graphs with a low-dimensional representation. We prove for instance that (not necessarily induced) subgraphs of the $d$-dimensional grid have asymptotic dimension $d$ (Theorem~\ref{thm:krlee}) and deduce that graphs of polynomial growth have bounded asymptotic dimension (Corollary~\ref{cor:polygrowth}). The assumption on the growth turns out to be sharp. We also make some observations about a characterisation of classes of bounded asymptotic dimension by forbidden (induced) subgraphs.

In Section~\ref{sec:pw} we prove that classes of graphs of bounded pathwidth have asymptotic dimension at most 1 (Theorem~\ref{thm:pw}) and deduce that classes of graphs of bounded layered pathwidth have asymptotic dimension at most 2 (Corollary~\ref{cor:lpwasdim}). This shows that any class of graphs excluding an apex-forest as a minor has asymptotic dimension at most 2 (Corollary~\ref{cor:apexforest}).

We conclude with some open problems in Section~\ref{sec:ccl}.

\section{Control functions and layerings}\label{sec:control}

We will need a quantitative version of a result of Brodskiy, Dydak, Levin
and Mitra~\cite{BDLM}, extending a result of Bell and
Dranishnikov~\cite{BD06}. In this section we introduce some
notation and state a quantitative version of (a special case of) their result. We then
explain the consequences of the result for graph layerings.

\subsection{Real projections, layerings and annuli}\label{sec:layer}

Given a metric space $(X,d)$, and a real $c>0$, a function $f: X\to
\mathbb{R}$ is \emph{$c$-Lipschitz} if for any
$x,y\in X$, $|f(x)-f(y)|\le c\cdot d(x,y)$ (such functions can be defined
between any two metric spaces, but here we will only consider
$\mathbb{R}$ as the codomain). A 1-Lipschitz
function $f: X\to \mathbb{R}$ is called a \emph{real projection} of
$X$. 

In the context of graphs, an interesting example of real projections comes from layerings. A 
\emph{layering} $L=(L_0,L_1,\ldots)$ 
of a graph $G$ is a partition of $V(G)$
into sets (called the \emph{layers}), such that for any edge $uv$ of $G$, $u$ and $v$ lie
in the same layer or in two consecutive layers. Note
that a layering  can also be seen as a function $L: V(G)\to
\mathbb{N}$ such that for any edge $uv$, $|L(u)-L(v)|\le 1$. In
particular a layering is a real projection.

A simple way to define a  real projection is by taking a rooted real projection. Given a metric
space $(X,d)$, and a point $x\in X$, the \emph{real projection rooted at $x$} is
the real projection $f$ defined by  $f(y)=d(x,y)$. 
Note that the triangle
inequality implies that this function is 1-Lipschitz, and thus a real projection. 
In an unweighted graph, the real projection rooted at a vertex $x$ 
immediately gives a layering  (where the layers are the
preimages of each element of $\mathbb{N}$ under the rooted real projection); this
layering corresponds to the levels of a Breadth-First Search (BFS) tree rooted in
$x$ (and is also known as a \emph{BFS-layering}). For general metric spaces (and
weighted graphs) it will be convenient to consider the following
continuous analogue of layers:
given a metric space $(X,d)$, a base point $x\in X$, and two reals
$0<a < b$ we denote by $A(a,b)$ the set of points $y$ such that $a\le
d(x,y) <b$. The sets $A(a,b)$ are called \emph{annuli}.  (For example, in a rooted real projection for ${\mathbb R}^2$ with Euclidean distance, the sets $(A,b)$ are the standard annuli; in higher dimensions they are spherical shells.)

A natural approach to defining a control function on a metric space is to fix a rooted real projection, and then divide the space into annuli of fixed width.  If 
we can find $n$-dimensional control functions for these annuli, then we can hope to stitch them together to find an $(n+1)$-dimensional control function for the whole space.  Theorem
\ref{thm:mthm49} will give us a tool for doing this.

\subsection{\texorpdfstring{$r$}\ -components and \texorpdfstring{$(r,s)$}\ -components}

Let $(X,d)$ be a metric space. Recall that a subset $S\subseteq X$ is
$r$-bounded if for any $x,x'\in S$, $d(x,x')\le r$.  Given a subset $A\subseteq X$, two points $x,x' \in A$
are \emph{$r$-connected} in $A$ if there are points $x=x_1,\dots,x_\ell=x'$ in $A$ such
that for any $1\le i \le \ell-1$, $x_i$ and $x_{i+1}$ are distance at most $r$ in $X$. 
A maximal set of
$r$-connected points in $A$ is called an \emph{$r$-component} of
$A$ (note that these components form a partition of $A$). Observe that in an unweighted graph $G$, the 1-components of a subset $U\subseteq V(G)$ of vertices are exactly the connected components of $G[U]$, the subgraph of $G$ induced by $U$ (see Figure~\ref{fig:rscomponent}, left, for an example of 2-components of a subset $A$ of vertices).

\begin{figure}[htb]
 \centering
 \includegraphics[scale=0.8]{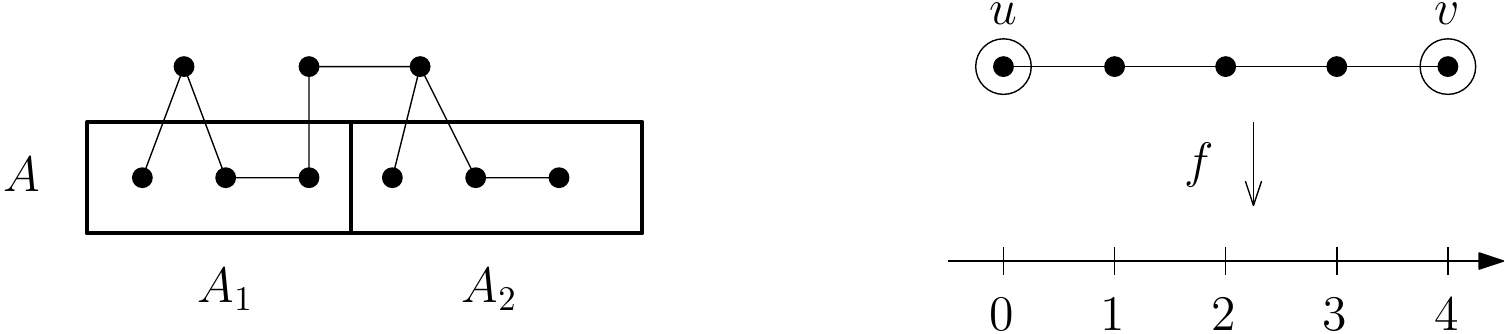}
 \caption{The two 2-components $A_1$ and $A_2$ of a subset $A$ (left), and an example of a set $B=\{u,v\}$ which has one 4-component but two $(4,3)$-components (right).}
 \label{fig:rscomponent}
\end{figure}

Recall that $D_X:\mathbb{R}^+\to
\mathbb{R}^+$ is an
$n$-dimensional control function for $X$ if for any $r>0$, $X$ has a cover
$\mathcal{U}=\bigcup_{i=1}^{n+1}\mathcal{U}_i$, such that each
$\mathcal{U}_i$ is $r$-disjoint and each element of $\mathcal{U}$ is
$D_X(r)$-bounded. Observe that, equivalently (using the terminology
introduced above), $D_X:\mathbb{R}^+\to
\mathbb{R}^+$ is an
$n$-dimensional control function for $X$ if for any $r>0$, $X$ has a
cover by $n+1$ sets whose $r$-components are $D_X(r)$-bounded. In this
section it will be convenient to work with this definition of control
functions.

\medskip

Let $(X,d)$ be a metric space and let $f: X\to \mathbb{R}$ be a real
projection of $X$. A subset $A\subseteq
V(G)$ is said to be \emph{$(r,s)$-bounded with respect to
$f$} if for all $x,x'\in A$
we have $d(x,x')\leq r$ and $|f(x)-f(x')|\leq s$
(when $f$ is clear from the context we often omit ``with respect to $f$'').
Two vertices $x,x'$ of $A$
are {\em $(r,s)$-connected} in $A$ if there are vertices $x=x_1,\dots,x_\ell=x'$ in $A$ such
that for any $1\le i \le \ell-1$, $\{x_i,x_{i+1}\}$ is $(r,s)$-bounded (i.e.\ $x_i$ and $x_{i+1}$ are at distance at most $r$ in $X$, and $|f(x)-f(x')|\leq s$). 
A maximal set of
$(r,s)$-connected vertices in $A$ is called an {\em $(r,s)$-component} of
$A$ (note that these components form a partition of $A$).

Note that by the definition of a real projection, any
$r$-bounded set is also $(r,r)$-bounded, and similarly being
$r$-connected is equivalent to being $(r,r)$-connected, and
an $r$-component is the same as an $(r,r)$-component. Observe that in Figure~\ref{fig:rscomponent}, right, the vertices $u$ and $v$ are 4-connected in $\{u,v\}$ but they are not $(4,3)$-connected in $\{u,v\}$ with respect to $f$, so this shows that the notions of $r$-components and $(r,s)$-components differ when $s<r$.

\subsection{Control functions for real projections}\label{sec:controlfun}

We have seen above that $D_X:\mathbb{R}^+\to
\mathbb{R}^+$ is an
$n$-dimensional control function for $X$ if for any $r>0$, $X$ has a
cover by $n+1$ sets whose $r$-components are $D_X(r)$-bounded. It will
be convenient to extend this definition to real projections,
as follows. For a metric space $X$ and a real projection $f:X\to
\mathbb{R}$, we say that  $D_f:\mathbb{R}^+\times \mathbb{R}^+\to
\mathbb{R}^+$ is an
\emph{$n$-dimensional control function for $f$} if for any $r,S>0$,
any $(\infty,S)$-bounded subset $A\subseteq X$ has a
cover by $n+1$ sets whose $r$-components are $D_f(r,S)$-bounded. We
say that the control function is \emph{linear} if there are constants
$a,b,c>0$ such that $D_f(r,S)\le ar+bS+c$ for any  $r,S>0$. We
say that the control function is a \emph{dilation}  if there are constants
$a,b>0$ such that $D_f(r,S)\le ar+bS$ for any  $r,S>0$.

The following is a special case of a result of Brodskiy, Dydak, Levin
and Mitra~\cite{BDLM}. We will need a more precise version of their
result (which follows from their proof, but is not stated explicitly
in their paper), so for completeness we give a proof of Theorem~\ref{thm:mthm49} in the
appendix.

\begin{thm}[Theorem 4.9 in~\cite{BDLM}]\label{thm:mthm49}
Let $X$ be a metric space and $f:X\to \mathbb{R}$ be a real projection
of $X$. If $f$ has an $n$-dimensional control function $D_f$, then
$X$ has an $(n+1)$-dimensional control function $D_X$ depending only on
$D_f$ and $n$. Moreover, if $D_f$ is linear then $D_X$ is linear and
if $D_f(r,S)\le ar+bS$, for some constants $a\ge 1,b\ge 1$ (i.e.\
$D_f$ is a dilation) then
$D_X(r)\le 20a(6a+b(n+4))\cdot r$ (i.e.\
$D_X$ is also a dilation).
\end{thm}

\medskip

We give the first applications of Theorem~\ref{thm:mthm49} for planar sets and
planar graphs. A \emph{planar set} is a geodesic metric set $(P,d_P)$
with an injective continuous map $\phi: P \to \mathbb{R}^2$. Fujiwara
and Papasoglu~\cite{FP20} proved that any annulus $A(a,b)$ in a rooted real projection of
a planar set or planar graph can be covered by two sets whose
$r$-components are $(10^5r+2(b-a))$-bounded.  
Note that each annulus $A(a,b)$ is
precisely an $(\infty,b-a)$-bounded set (with respect to the rooted real projection), so
the result of Fujiwara
and Papasoglu~\cite{FP20} can be rephrased as follows.

\begin{lemma}[Lemma 4.4 in~\cite{FP20}]\label{lem:fp20}
Suppose $(P,d_P)$ is a planar set or a connected planar graph, and let $L$ be a
rooted real projection of $(P,d_p)$.
Then $L$ has a 1-dimensional control
function $D_L$ with $D_L(r,S)\le 10^5r+2S$.
\end{lemma}

Observe the following.

\begin{obs}\label{obs:compo}
A function $D$ is an $n$-dimensional control function for a graph $G$ if and only $D$ is an
$n$-dimensional control function for each of the connected components of $G$.
\end{obs}

Using Lemma~\ref{lem:fp20}, Fujiwara
and Papasoglu~\cite{FP20} proved that planar sets and planar graphs
have Assouad-Nagata dimension at most 3. Applying
Theorem~\ref{thm:mthm49} and Observation~\ref{obs:compo} together with Lemma~\ref{lem:fp20}, we obtain the following immediate improvement over their result (as stated in the introduction).

\begin{thm}\label{thm:planar}
Planar sets and planar graphs
have Assouad-Nagata dimension at most 2.
\end{thm}

This gives a positive answer to Question 5.1 in~\cite{FP20}.
Note that this result is best possible, as the class of 2-dimensional
grids has asymptotic dimension 2.
As mentioned in the introduction, Theorem~\ref{thm:planar} was also proved independently by J\o rgensen and Lang in~\cite{JL20}.

In Section~\ref{sec:k3p} we will extend this result from planar graphs
to graphs avoiding $K_{3,p}$ as a minor. This requires additional work, some of which will be crucial to the remainder of the paper. In
Section~\ref{sec:surfaces} we will explain the consequences for graphs
embeddable on surfaces of bounded genus (and for compact Riemannian
surfaces). For all the remaining results, it will be convenient to work
not in the metric induced by the original metric space $X$, but in the
intrinsic metric in the subspace defined by an annulus. We now explain how to prove
a variant of Theorem~\ref{thm:mthm49} in this setting.

\subsection{Intrinsic control of real projections}\label{sec:intrinsic}

From now on, all the metric spaces under consideration are finite
weighted or unweighted graphs, together with their shortest path metric (we will also consider compact
Riemannian surfaces but all results about these metric spaces will
be deduced from our results on finite weighted graphs). 
When considering an induced subgraph
$H$ of $G$, we will call the shortest path metric of $H$
\emph{intrinsic}, and denote it $(H,d_H)$, to make a clear distinction
from the space $(H,d_G)$ where the metric under consideration is the shortest path metric of $G$.

Let $G$ be a (weighted or unweighted) graph and $L:V(G)\to
\mathbb{R}^+$ be a real projection of $G$ (in the remainder $L$ will always be a layering in the case
of unweighted graphs, and a rooted real projection, in the case of weighted graphs).

A function $D_L:\mathbb{R}^+\times \mathbb{R}^+\to \mathbb{R}^+$ is an \emph{$n$-dimensional
  intrinsic control function} for $L$ if for all $r,S>0$, the
graph $H$ induced by any $(\infty,S)$-bounded set of $G$ (with respect
to $L$ and in the metric $d_G$) can be covered by $n+1$ sets whose $r$-components are
$D_L(r,S)$-bounded (where the definition of $r$-components and
$D_L(r,S)$-bounded are in the intrinsic metric $d_H$).

\smallskip

As before, we say that an intrinsic control function $D_L$ for a
real projection $L$ is \emph{linear}
if it is of the form $D_L(r,S)=ar+bS+c$, for some constants
$a,b,c>0$. We also say that $L$ is \emph{a dilation}
if it has an intrinsic control function of form $D_L(r,S)= ar+bS$, for some constants
$a,b>0$.

\smallskip

We now prove that intrinsic control functions can be transformed into (classical)
control functions.

\begin{lemma}\label{lem:intrinsic}
Let $G$ be a  (weighted or unweighted) graph and $L:V(G)\to
\mathbb{R}^+$ be a real projection of $G$, such that $L$ admits an
$n$-dimensional intrinsic control function $D$. Then $D'(r,S):=D(r,S+2r)$ is an
$n$-dimensional control function for 
$L$. In particular, if $D$ is linear, then $D'$ is also linear, and if
$D$ is a dilation, then $D'$ is a dilation.
\end{lemma}

\begin{proof}
  Let $X$ be a  $(\infty,S)$-bounded
  subset of vertices of $G$. It follows that there is some $a>0$
  such that for any $x\in X$, $L(x)\in [a,a+S]$. Fix some $r>0$, and
  denote by $X^+\supseteq X$ the preimage of the interval $[a-r,a+S+r]$ under
  $L$. Let $H^+:=G[X^+]$ be the subgraph of $G$ induced by $X^+$. Then by definition of
  $D$, $(H^+,d_{H^+})$ has a cover by $n+1$ sets
  $(U_i^+)_{i=1}^{n+1}$, whose $r$-components are
  $D(r,S+2r)$-bounded. For each $1\le i \le n+1$, let $U_i$
  be the restriction of the family ${U}_i^+$ to $X$. It
  follows that the sets $({U}_i)_{i=1}^{n+1}$ cover $X$. Consider now an
  $r$-component $C$ of some set ${U}_i$, for $1\le i \le n+1$ (where
  the distances in the definition of $r$-component are with respect to
  the metric $d_G$). For any $u,v\in C$, there are $u=u_0,u_1,\ldots,u_t=v$ in $U_i$ such for any
 $0\le j \le t-1$, $d_G(u_j,u_{j+1})\le r$. It follows that $u_j$
and $u_{j+1}$ are connected by a path (of length at most $r$) in $G$, and thus also in 
$H^+=G[X^+]$ by definition of $X^+$. Hence, $u_j$
and $u_{j+1}$ lie in the same $r$-component of $U_i^+$ (where all the distances
involved in the definitions are with respect to $d_{H^+}$). Since
all $r$-components of $U_i^+$ (with respect to $d_{H^+}$) are
$D(r,S+2r)$-bounded with respect to $d_{H^+}$, they are also
$D(r,S+2r)$-bounded with respect to $d_G$.
This shows that $D'(r,S):=D(r,S+2r)$ is an
$n$-dimensional control function for 
$L$.  
  \end{proof}

\subsection{Layerability}

In this section, we will consider graph layerings (or more generally real projections) such that any constant number of consecutive layers induce a graph from a "simpler" class of graphs. For instance, any $d$-dimensional grid has a layering in which each layer induces a $(d-1)$-dimensional grid, and moreover any constant number of consecutive layers induce a graph that is "$(d-1)$-dimensional in the large scale".

  Given a (weighted or unweighted) graph class $\mathcal{C}$ and a sequence of (weighted or unweighted) graph classes
$\mathcal{L}=(\mathcal{L}_i)_{i\in \mathbb{N}}$ with $\mathcal{L}_0\subseteq \mathcal{L}_1\subseteq\mathcal{L}_2\subseteq\cdots$, we say
that $\mathcal{C}$ is \emph{$\mathcal{L}$-layerable} if there is a
function $f:\mathbb{R}^+\to \mathbb{N}$ such that any graph  $G\in \mathcal{C}$
has a real projection $L: V(G)\to \mathbb{R}$ such that for any
$S>0$, any $(\infty,S)$-bounded set in $G$ (with respect to $L$) induces a graph from
$\mathcal{L}_{f(S)}$. If there is a constant $c>0$ such that $f(S)\le
cS$ for any $S>0$, we say that
$\mathcal{C}$ is \emph{$c$-linearly $\mathcal{L}$-layerable}.

\begin{thm}\label{thm:bd}
  Suppose that all the classes in a sequence $\mathcal{L}=(\mathcal{L}_1, \mathcal{L} _2,\ldots)$ of (weighted or unweighted) graph
  classes have asymptotic dimension at most $n$, and let $\mathcal{C}$
  be an $\mathcal{L}$-layerable class of (weighted or unweighted) graphs. 
Then $\mathcal{C}$ has asymptotic dimension at
  most $n+1$. Let $a,b,c\ge 1$ and $d\ge 0$ be some constants, and assume moreover that the class $\mathcal{C}$ is $c$-linearly
  $\mathcal{L}$-layerable. If each class $ \mathcal{L}_i$ has an
  $n$-dimensional control function $D_i(r)\le a r+bi+d$, then $\mathcal{C}$ has asymptotic dimension at
  most $n+1$ of linear type. If moreover $d=0$, then
  $\mathcal{C}$ has Assouad-Nagata dimension at
  most $n+1$, with an $(n+1)$-dimensional control function
  $D_{\mathcal{C}}(r)\le 20(6(a+2bc)^2+(a+2bc)bc(n+4)) \cdot r$.
\end{thm}

\begin{proof}
For any $i\ge 1$, Let $D_i$ be an $n$-dimensional control function for the graphs of
$\mathcal{L}_i$.
Let $G\in \mathcal{C}$, and let $L: V(G)\to \mathbb{R}$ be a
real projection of $G$ such that for any
$S>0$, any $(\infty,S)$-bounded set in $G$ (with respect to $L$) induces a graph of
$\mathcal{L}_{f(S)}$.
Note that $D(r,S):=D_{f(S)}(r)$ naturally defines an $n$-dimensional
intrinsic control function for $L$. By Lemma~\ref{lem:intrinsic}, $D'(r,S):=D(r,S+2r)$ is an
$n$-dimensional control function for $L$. By Theorem~\ref{thm:mthm49}, $G$ admits an $(n+1)$-dimensional
control function $D_G$ such that $D_G(r)$ only depends on $r$, $D$ and
$n$. This shows that $\mathcal{C}$ has asymptotic dimension at
  most $n+1$. 

If the class $\mathcal{C}$ is $c$-linearly
$\mathcal{L}$-layerable, then it holds that $D'(r,S)=D(r,S+2r)=D_{c(S+2r)}(r)$. If
moreover $D_i(r)\le a r+bi+d$, then 
\[
D'(r,S)\le a r+b(c(S+2r))+d= (a+2bc) r+bcS+d,
\]
which means that $D'$ is linear. In this case it follows from
Theorem~\ref{thm:mthm49} that $D_G$ is linear, and thus $\mathcal{C}$ has asymptotic dimension at
most $n+1$ of linear type.

If moreover $d=0$,  we have $D'(r,S)\le
  (a+2bc)r+bcS$. By Theorem~\ref{thm:mthm49},
$D_G(r)\le 20(6(a+2bc)^2+(a+2bc)bc(n+4))\cdot r$ (i.e.
$D_G$ is also linear). In particular, $\mathcal{C}$ has Assouad-Nagata
dimension at
most $n+1$.
\end{proof}

Informally, if all the graphs in a class $\mathcal{C}$ have a layering
such that any constant number of layers induce a graph from a class of
asymptotic dimension $n$, then $\mathcal{C}$ has asymptotic dimension
at most $n+1$.
The following immediate corollary is a simple rephrasing of
Theorem~\ref{thm:bd}, avoiding the terminology of layerability.

\begin{corollary}\label{cor:bd}
Let $\mathcal{C}$ be a class of (weighted or unweighted) graphs, and
let $a\ge 1,b\ge 1$ and $d\ge0$ be some real numbers. Suppose that there exists
an integer $n$ such that any graph $G$
of $\mathcal{C}$ has a real projection $L:V(G)\to \mathbb{R}$
such that any $(\infty,S)$-bounded set (with respect to $L$) induces a
graph with an $n$-dimensional control function $D_{G,S}(r)\le ar+bS+d$ for any $r,S>0$. Then $\mathcal{C}$ has asymptotic dimension at
  most $n+1$ of linear type. If moreover $d=0$, then $\mathcal{C}$ has
Assouad-Nagata dimension at
  most $n+1$, with an $(n+1)$-dimensional control function
  $D_{\mathcal{C}}(r)\le 20(6(a+2b)^2+(a+2b)b(n+4))\cdot r$.
\end{corollary}

\begin{proof}
Simply define $ \mathcal{L}_i$ as the class of graphs having an
  $n$-dimensional control function $D_i(r)\le a r+bi+d$. Then the class $\mathcal{C}$ is $1$-linearly
  $\mathcal{L}$-layerable, with $\mathcal{L}:=(\mathcal{L}_i)_{i\ge
    0}$. The result thus follows from Theorem~\ref{thm:bd} by taking $c=1$.
\end{proof}

\section{\texorpdfstring{$K_{3,p}$}\ -minor free graphs}\label{sec:k3p}

\subsection{Terminology}

We recall that in this section all graphs are finite.
For two  subsets $A,B$ of vertices of $G$, we define
$$d_G(A,B):=\min\{d_G(u,v)\,|\,(u,v)\in A\times B\}.$$ 
The \emph{diameter} of a subset $U$ of vertices of $G$, denoted by
$\diam_G(U)$, is the maximum over the distances in $G$ between two vertices of
$U$ (we have called it the \emph{weak diameter} in
Section~\ref{sec:clustered}, which is the terminology in theoretical
computer science, but in this section we will refer to it simply as the
diameter, which is the usual terminology in the theory of metric
spaces).

\smallskip

A path between two vertices $u$ and $v$ in a weighted graph $G$ is \emph{geodesic} if its
length is precisely $d_G(u,v)$ (i.e.\ the path is a (weighted) shortest path
between its endpoints). 

\smallskip

We say that a subset $S$ of vertices of $G$ is
\emph{connected} if the subgraph of $G$ induced by $S$, denoted by
$G[S]$, is connected.

\subsection{Fat bananas}
For two integers $p$ and $q$, a \emph{$q$-fat $p$-banana}
in a weighted
graph
$G$ is a pair of connected subsets $A$ and $B$ of vertices of $G$ with $d_G(A,B)\geq q$,
together
with $p$ geodesic paths between $A$ and $B$, such that $d_G(x,y)\geq
q$ if $x,y$ belong to different paths (see Figure~\ref{fig:fatbanana}, left). Note that for $(0,1]$-weighted graphs, the assumption that $A$ and $B$ are at distance at least $q$ implies that the geodesic paths contain at least $q$ vertices. 

\medskip

\begin{figure}[htb]
 \centering
 \includegraphics[scale=0.7]{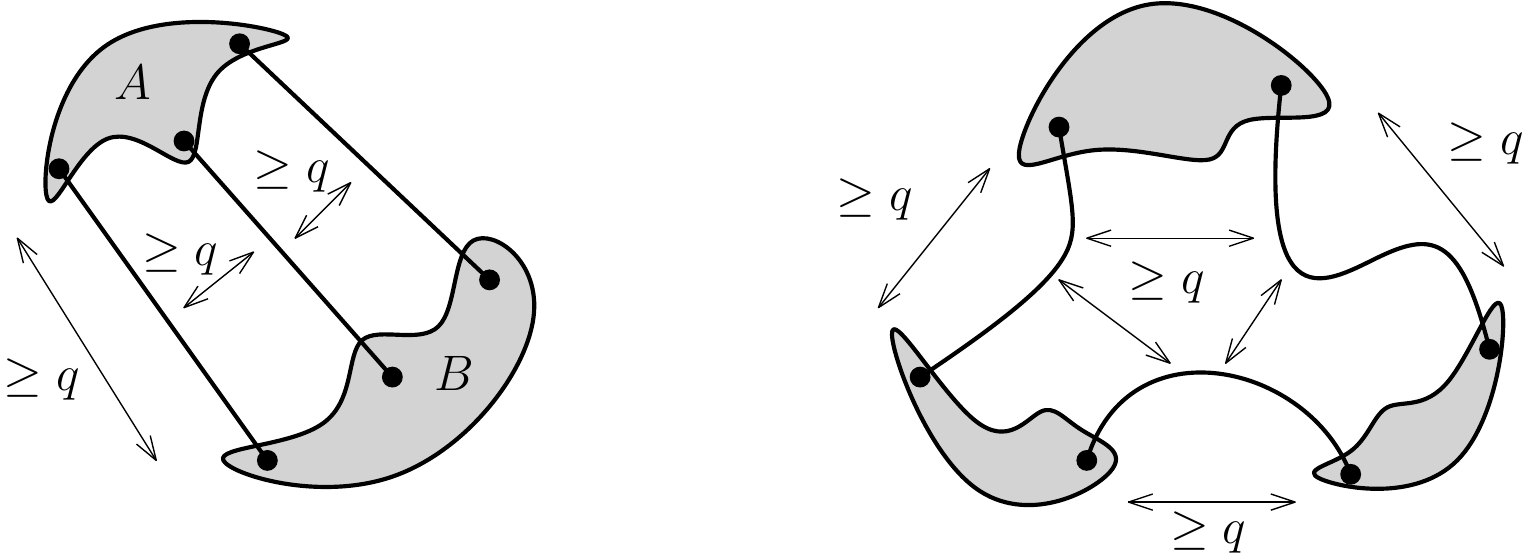}
 \caption{A $q$-fat 3-banana (left), and a $q$-fat $K_3$-minor (right).}
 \label{fig:fatbanana}
\end{figure}

In this section, we will prove a technical lemma that generalises the following (rephrased) result of Fujiwara and Papasoglu \cite{FP20}.
\begin{thm}[Theorem 3.1 in \cite{FP20}]\label{thm:FP20}
Let positive integers $q$ and $r$ be given. Any graph $G$ with no
$q$-fat $3$-banana has a cover by $2$ sets whose $r$-components are  $(10^5r)$-bounded.
\end{thm}

We now extend their result to graphs with no $q$-fat $p$-bananas.

\begin{lemma}
\label{lem:dim_theta_free}
Let positive integers $p\ge 2$, $q\ge 1$ and a real number $r>0$ be given. Then,
any weighted graph $G$ with no $q$-fat $p$-banana has a cover by 2
sets whose $r$-components are $(5r+3q)p$-bounded.
\end{lemma}
\begin{proof}
Let $G$ be a weighted
graph with no $q$-fat $p$-banana. We assume that the graph $G$ is
connected, otherwise we can consider each connected component
separately. We fix a root vertex $v\in G$, and recall that an
\emph{annulus} is a vertex set of the form $A(a,b)=\{u\,|\, a\le d_G(u,v)<b\}$, for some $0<a<b$. 
For any integer $k\geq 1$, let $A_k:=A(kr, (k+1)r)$. For any integer $k$, we say that the annuli $A_k$ and $A_{k+1}$ are \emph{consecutive}.

Let $k_0$ be the smallest integer such that $k_0r\ge r+q$, and
assume without loss of generality that $k_0$ is odd (if not,
interchange even and odd in the remainder of the proof). Note that by definition,
$(k_0-1)r\le r+q$, and thus $k_0r\le 2r+q$. 
Let $A_0$ be the set of vertices at distance less than $k_0r$ from $v$. Note that $A_0$ has diameter
at most $2k_0r\le 4r+2q\le (5r+3q)p$. It remains to cover the vertices at distance at least
$k_0r$ from $G$, that is the vertices in $\bigcup_{k\ge k_0} A_k$.

We define the first covering set as $C_0:=A_0\cup \bigcup \{A_k\,|\, k\ge k_0, \,k\text{
  even}\} $ and the second covering set as $C_1:=\bigcup \{A_k\,|\, k\ge k_0, \,k\text{
  odd}\} $. Note that these two sets clearly cover $V(G)$, so we only
need to show that each $r$-component in $C_0$ or $C_1$ is
$(5r+3q)p$-bounded. By the definition of the annuli $A_k$, observe
that each $r$-component of $C_0$ or $C_1$ is contained in $A_0$ or in
some annulus $A_k$ with $k\ge k_0$. 
Since $A_0$ is $(5r+3q)p$-bounded, it thus suffices to show that for
any $k\ge k_0$, each
$r$-component of $A_k$ is also $(5r+3q)p$-bounded.

Fix some $k\ge k_0$, and let $C$ be an $r$-component of $A_k$. Assume for the sake of contradiction that there exist
$x$ and $y$ in $C\subseteq A_k$, with $d_G(x,y)\ge (5r+3q)p$. By
definition of $C$, we may find $x_1,\dots,x_\ell\in C$, where
$x_1=x$, $x_\ell=y$ and
$d_G(x_i,x_{i+1})\leq r$ for $1\le i\le \ell-1$.

We now define an increasing function $\iota : \{1,\ldots
,p\}\mapsto \N$, as follows: 
$\iota(1)= 1$, $\iota(p)=\ell$, and for any $1\le i\le p-2$, $\iota(i+1)$ is one plus the
largest integer $j\ge \iota(i)$ such that $d_G(x_{\iota(i)},x_{j})\leq
4r+3q$. By the definition of $\iota$ (in particular the fact that $\iota$ is increasing), we have
$d_G(x_{\iota(i)},x_{\iota(j)})> 4r+3q$ for any $1\le i<j \le
p-1$.

We claim that we also have $d_G(x_{\iota(i)},x_{\iota(p)})>
4r+3q$ for any $1\le i \le p-1$.
Since for any $1\le i \le p-2$ we have
\begin{align*}
d_G(x_{\iota(i+1)},x_{\iota(i)})& \le 
d_G(x_{\iota(i+1)},x_{\iota(i+1)-1})+d_G(x_{\iota(i+1)-1},x_{\iota(i)})\\
  &\leq 
    r+4r+3q\\
  &\le 5r+3q,
\end{align*}
we obtain that for any $1\le i \le p-1$, $d_G(x_{1},x_{\iota(i)})\le
(5r+3q)i$.
As a consequence,
\begin{align*}
d_G(x_{\iota(i)},x_\ell)&\geq d_G(x_1,x_\ell)-d_G(x_1,x_{\iota(i)})\\
&\geq (5r+3q)p-(5r+3q)i \ge 5r+3q> 4r+3q,
\end{align*}
which concludes the proof of our claim, and shows that we have
$d_G(x_{\iota(i)},x_{\iota(j)})> 4r+3q$ for any $1\le i<j \le
p$.

\medskip

For each $1\le i \le p$, consider a shortest path $P_i$ from $x_{\iota(i)}$
to the root $v$. We now define a $q$-fat $p$-banana $\Theta$ as follows:
\begin{itemize}
\item the set $A$ is the subset of vertices of $\bigcup_{1\le i \le p}
  V(P_i)$ 
  at distance at most $kr-(r+q)$ from $v$ (note that since
  $k\ge k_0$, we have $kr\ge r+q$ and the distance $kr-(r+q)$ is
  well defined);
\item the set $B$ is the set $x_1,x_2,\ldots,x_\ell$, together with
  shortest paths (of length at most $r$) between consecutive vertices
  $x_i,x_{i+1}$ for any $1\le i \le \ell-1$;
  \item each of the $p$ geodesic paths of $\Theta$ is a minimal $A$--$B$ path in $P_i$ (i.e. this subpath contains vertices at distance at least $kr-(r+q)$ from $v$).
\end{itemize}

Note that it immediately follows from the definition of $\Theta$ that $A$ and $B$
induce connected subgraphs of $G$.

Since the vertices $x_1,x_2,\ldots,x_\ell$ are in $A_k$, they are at
distance at least $kr$ from $v$, and the vertices on the shortest
paths connecting two consecutive $x_i,x_{i+1}$ are at distance at
least $kr-r$ from $v$. It follows that $B$ is at distance at least
$(k-1)r$ from $v$ and thus at distance at least
$(k-1)r-(kr-(r+q))=q$
from $A$, as desired.

It remains to check that any two vertices $x,y$ lying on two different paths of $\Theta$ are at distance at least $q$ apart in $G$. By definition of $A$, we may assume that $d_G(x_{\iota(i)},x) \le (k+1)r - (kr - (r+q)) \le 2r+q$ and similarly, for some $i\neq j$, that $d_G(x_{\iota(j)},y) \le 2r+q$. So if $d_G(x,y)\le q$, we obtain $d_G(x_{\iota(i)}, x_{\iota(j)}) \le 4r+3q$, which is a contradiction.

We have proved that $\Theta$ is a $q$-fat $p$-banana, which contradicts
our assumption that $G$ does not contain any $q$-fat $p$-banana. This
shows that each $r$-component of $C_0$ or $C_1$ is $(5r+3q)p$-bounded, as desired.
\end{proof}

Note that a graph containing a 2-fat 2-banana also contains a
chordless cycle of length at least 4. It follows that
chordal graphs have no 2-fat 2-bananas. Setting $p=2, q=2$ in
Lemma~\ref{lem:dim_theta_free}, we obtain that for any $r>0$, chordal
graphs have a cover by 2 sets whose $r$-components are
$(10r+12)$-bounded. Recall that any
$r$-ball intersects at most one $2r$-component of a given set, so we
obtain a cover of any chordal graph by $(20r+12)$-bounded
sets such that any
$r$-ball intersects at most two sets of the cover.
For unweighted chordal graphs, we can assume that
$r\ge 1$ is an integer and thus $20r+12\le 32r$. Hence,
we obtain a $(32,2)$-weak sparse partitioning scheme
for chordal graphs. Note that Filtser~\cite{Fil20} obtained a
$(24,3)$-weak partition scheme for this class of graphs.

\subsection{Fat minors}\label{sec:fatminor}

A (weighted or unweighted) graph $G$ contains a graph $H$ as a \emph{minor} if $G$ contains $|V(H)|$
vertex-disjoint subsets $\{T_v\,|\,v \in V(H)\}$, each inducing a
connected subgraph in $G$, and such that for every edge $uv$ in $H$,
$T_u$ and $T_v$ are connected by an edge in $G$. 
We will also need the
following interesting metric variant of minors: for some integer $q\ge 1$, a
weighted graph $G$ contains a graph $H$ as a \emph{$q$-fat minor} if $G$ contains $|V(H)|$
vertex-disjoint subsets $\{T_v\,|\,v \in V(H)\}$ such that
\begin{itemize}
\item each subset $T_v$ induces a connected subgraph of $G$;
  \item any two sets $T_u$ and $T_v$ are at distance at least $q$
    apart in $G$;
\item for every edge $uv$ in $H$,
  $T_u$ and $T_v$ are connected by a path $P_{uv}$ (of length
  at least $q$) in $G$, such that 
  \begin{itemize}
\item for any pair of distinct edges $uv$ and $xy$ of $H$ (possibly sharing a
  vertex), the paths
  $P_{uv}$ and $P_{xy}$ are at distance at least $q$ in $G$, and 
  \item for any edge $uv$ in $H$ and any vertex $w$ distinct from $u$ and $v$, $P_{uv}$ is at distance at least $q$ from $T_w$ in $G$.  
  \end{itemize}
\end{itemize}
An example of $q$-fat $K_3$-minor is depicted in Figure~\ref{fig:fatbanana}, right, where the last property (stating that sets $T_w$ are far from paths $P_{uv}$) is not mentioned explicitly, for the sake of readability.

\medskip

Note that if $G$ contains $H$ as a $q$-fat minor, then it also contains $H$
as a minor. 
Fat minors are related to fat bananas by the following
simple lemma (we recall that we have chosen
the interval
$(0,1]$ for simplicity but it can be replaced with any interval $(0,C]$, $C>0$, without
affecting our results).

\begin{lemma}\label{lem:thetaminor}
Let $G$ be a $(0,1]$-weighted graph. If $G$ does not
contain the complete bipartite graph $K_{2,p}$ as a $q$-fat minor,
then $G$ contains no $(3q+2)$-fat $p$-banana.
\end{lemma}
\begin{proof}
We prove the contrapositive.
Consider a $(3q+2)$-fat $p$-banana in $G$, i.e.\ connected sets $A$ and $B$ at distance at least $3q+2$
apart in $G$,
together with $p$ geodesic paths $P_1,\ldots,P_p$ connecting $A$ and
$B$, such that any pair of paths $P_i,P_j$  is at distance at least
$3q+2$ apart in $G$. 

For any $1\le i \le p$, we partition the vertices of each $P_i$ into three parts $P_{i,A}, P_{i,0}$ and $P_{i,B}$. The subpath
$P_{i,A}$ of $ P_i$ consists of all vertices at distance at most $q$ from $A$ and similarly for $P_{i,B}$. Since $3q+2>1$ and since the weight of any edge is at most 1, the set $P_{i,0}$ is non-empty. Moreover, each $P_{i,0}$ is at distance at least $q$ from $A$ and $B$, and the paths $P_{i,A}$ and $P_{i,B}$ are at least $q$ apart as well. Finally, for $i\neq j$ the lower bound of $3q+2$ on the distance between vertices in $P_i$ and $P_j$ remains. This implies
that the sets $A,B,P_1^0,P_2^0,\ldots,P_p^0$, together with the paths
$P_1^A,P_1^B,\ldots ,P_p^A,P_p^B$ define a $q$-fat $K_{2,p}$ minor in
$G$.
\end{proof}

This will not be needed in the remainder, but we mention nevertheless that for $(0,1]$-weighted graphs, any $q$-fat $H$-minor can be transformed into a $(q/3-1)$-fat $H$ minor in which all paths are geodesic. Indeed, for an edge $uv$ of $H$ we can grow $T_u$ along $P_{uv}$ until the first vertex at graph distance at most $q/3$ from $T_v$ and then add in a geodesic $Q_{uv}$ to $T_v$. If we do this procedure once for each edge, then the resulting geodesics are still $q/3$ apart, and the vertices added to $T_u$ will be at least $q/3-1$ from each vertex in $T_v$.

\medskip

Let $G$ be a graph and $H$ be an induced subgraph of $G$.
We recall that in Section~\ref{sec:intrinsic} we have defined the
intrinsic shortest path metric $d_H$ in $H$, by opposition with the
$(H,d_G)$, which is the graph $H$ equipped with the shortest path
metric in $G$.  
we say that $H$ contains $\Theta$ as an \emph{intrinsic}
fat banana, or contains $H$ as an \emph{intrinsic}
fat minor, if the definition of fat banana or fat minor is with
respect to the metric $d_H$.

\medskip

Given a graph $H$, the graph $H^*$ is obtained from $H$ by adding a
universal vertex (i.e.\ adding a vertex and joining it to all the
vertices of $H$). We now prove that if a graph $G$ avoids $H^*$ as a
minor, then the subgraph of $G$ induced by any annulus in $G$ avoids
$H$ as an intrinsic fat minor of comparable or greater length.

\begin{lemma}\label{lem:layerqfat}
Let $H$ be a graph, and let $G$ be a weighted graph with no $H^*$ minor. 
For any vertex $v\in V(G)$, and any real numbers $0< s < t$ and $q\ge 2(t-s)$, the subgraph $G'$ of $G$ induced by the vertices 
$\{u \in V(G)\,|\, s\le d_G(u,v)<t\}$ 
does not contain $H$ as an intrinsic $q$-fat minor.
\end{lemma}
\begin{proof}
Let $A=\{u
\in V(G)\,|\,  d_G(u,v)<s\}$ and $B=\{u
\in V(G)\,|\, s\le d_G(u,v)<t\}$. Assume for the sake of contradiction
that $G'=G[B]$ contains $H$ as an intrinsic $q$-fat minor. Let $\{T_u\,|\,u\in V(H)\}$ and
$\{P_{uw}\,|\,uw \in E(H)\}$ be as in the definition of a $q$-fat
minor. For each $u\in V(H)$, let $P_u$ be a shortest path from $T_u$
to $v$ in $G$ and let $P_u^+=P_u
\cap B$. Note that the length of $P_u^+$ is less than $t-s$. For any $u\in V(H)$, set $T_u^+:=T_u \cup V(P_u^+)$.

Each of the sets $T_u^+$ is connected and disjoint from $A$, and two
sets $T_u^+$ and $T_w^+$ are at distance greater than $q-2(t-s)\ge 0$ apart in $G'$ (in
particular the sets are pairwise disjoint). The set $A$ is connected and has an edge to $T_u^+$ for all $u\in V(H)$.
For any edge $uw$ in $H$, let $P_{uw}^+$ be a minimum subpath of $P_{uw}$ between $T_u^+$ and $T_w^+$. 
Since each path $P_{uw}$ is at distance at least $q>t-s$ from all the
sets $T_x$ with $x \not\in \{u,w\}$, the subpaths $P_{uw}^+$ are
disjoint from all sets $T_x^+$ with $x \not\in \{u,w\}$.
 Moreover, since the paths $P_{uw}$ are pairwise vertex-disjoint, and disjoint from $A$, the subpaths $P_{uw}^+$ are also pairwise vertex-disjoint and disjoint from $A$.
It follows that $G$ contains
$H^*$ as a minor, which is a contradiction.
\end{proof}

\subsection{\texorpdfstring{$K_{3,p}$}\ -minor free graphs}

Recall that Fujiwara and Papasoglu~\cite{FP20} proved that unweighted planar
graphs have Assouad-Nagata dimension at most 3. We now prove that the
dimension can be reduced to 2, and the class of graphs can be extended
to all graphs avoiding $K_{3,p}$ as a minor, for some fixed integer
$p\ge 2$. Moreover the result below holds for weighted graphs (this
will be used to derive our result on compact Riemannian surfaces in Section~\ref{sec:surfaces}).

\begin{thm}\label{thm:k3p}
For any integer $p\ge 3$, any weighted graph with no
$K_{3,p}$-minor has a linear
2-dimensional control function $D_p$. Moreover, if $G$ is
$[\epsilon,\infty)$-weighted, for some $\epsilon>0$, then $D_p(r)\le c \cdot p^2r$, for some
constant $c=c(\epsilon)>0$. In particular, the class of weighted graphs with no $K_{3,p}$-minor
has asymptotic dimension at most 2 of linear type, and the class of unweighted graphs with no $K_{3,p}$-minor
has Assouad-Nagata dimension at most 2. 
\end{thm}

\begin{proof}
  By Observation~\ref{obs:compo}, it suffices to prove the result for
  connected graphs, and by Observation~\ref{obs:subv}, we can restrict ourselves to $(0,1]$-weigthed graphs (since the class of $K_{3,p}$-minor free graphs is closed under taking subdivisions).
Consider a $(0,1]$-weighted connected graph $G$ with no
$K_{3,p}$-minor, and let $L: V(G)\to \mathbb{R}$ be a rooted real projection of $G$, i.e.\ we fix a
vertex $v$ and for any $u\in V(G)$ we set $L(u):=d_G(u,v)$. 
We fix $S>0$ and we 
consider a subgraph $H$ of $G$ induced by some $(\infty,S)$-bounded
subset of $V(G)$ (with respect to $L$).

If $v\in V(H)$, then $(H,d_H)$ is $2S$-bounded (since all the vertices
of $V(H)$
are at distance at most $S$ from $v$ in $H$), and thus all $r$-components (with respect to $d_H$) of
$V(H)$ are $2S$-bounded in $d_H$.
If $v\not\in V(H)$ then by Lemma~\ref{lem:layerqfat}, $H$ does not contain
$K_{2,p}$ as an intrinsic
$2S$-fat minor. As a consequence, it follows from
Lemma~\ref{lem:thetaminor} that $H$ does not contain any
intrinsic $(6S+2)$-fat $p$-banana. By
Lemma~\ref{lem:dim_theta_free}, $(H,d_H)$ has a
1-dimensional function $D_H$ with $D_H(r)\le (5r+18S+6)p$.

We obtain that in any case, any $(\infty,S)$-bounded
subset of $V(G)$ (with respect to $L$) induces a
graph with a 1-dimensional function $D_S$ with $D_S(r)\le (5r+18S+6)p$ for any $r>0$.
By
Corollary~\ref{cor:bd}, $G$ has a linear 2-dimensional control $D_G$
depending only on $p$, and thus the class of $(0,1]$-weighted graphs
with no $K_{3,p}$-minor has asymptotic dimension at most 2 of linear type.

If $G$ is $[\epsilon,1]$-weighted, for $\epsilon>0$, then we can
assume that $r\ge \epsilon$ and thus $D_S(r)\le (5r+18S+6)p\le
(5+6/\epsilon)pr+18pS$. 
By Corollary~\ref{cor:bd},  the class of $K_{3,p}$-minor free
$[\epsilon,1]$-weighted graphs
has Assouad-Nagata dimension at most 2, with a 2-dimensional control function $D_p(r)\le c\cdot p^2r$, for some
constant $c=c(\epsilon)>0$. The case of unweighted graphs corresponds
to $\epsilon=1$.\end{proof}

\section{Surfaces}\label{sec:surfaces}

\subsection{Graphs on surfaces}

In this section, a \emph{surface} is a non-null  compact connected
2-dimensional manifold without boundary. A surface can be orientable
or non-orientable. The \emph{orientable
  surface of genus~$h$} is obtained by adding $h\ge0$
\emph{handles} to the sphere; while the \emph{non-orientable
  surface of genus~$k$} is formed by adding $k\ge1$
\emph{cross-caps} to the sphere. By the Surface classification
theorem, any surface is of one of these two types (up to
homeomorphism). The {\em Euler genus} of a surface
$\Sigma$ is defined as twice its genus if $\Sigma$ is orientable, and as
its non-orientable genus otherwise.

It is well known that any graph embeddable in the plane (or
equivalently the sphere) excludes $K_{3,3}$ as a minor. More generally,
for any integer $g\ge 0$, no graph embeddable on a surface of Euler genus $g$ contains
$K_{3,2g+3}$ as a minor\footnote{Since  the
class of graphs embeddable in a surface of Euler genus $g$ is closed under taking minors, it suffices to show that
$K_{3,2g+3}$ cannot be embedded on a surface of Euler genus $g$. This
is a simple consequence of Euler's Formula (see Proposition 4.4.4 in~\cite{MoTh}).}.
We thus obtain the result below as an immediate consequence of Theorem~\ref{thm:k3p}.

\begin{corollary}\label{cor:genus2}
For any integer $g\ge 0$, any weighted graph embeddable on a
surface of Euler genus $g$ has a linear
2-dimensional control function $D_g$ depending only on $g$. Moreover, if $G$ is
$[\epsilon,\infty)$-weighted, for some real $\epsilon>0$, then $D_g(r)\le c \cdot g^2r$, for some
constant $c=c(\epsilon)>0$. In particular, the class of
weighted graphs embeddable on a surface of Euler genus $g$ 
has asymptotic dimension at most 2 of linear type, and the class of unweighted graphs embeddable on a surface of Euler genus $g$ has Assouad-Nagata dimension at most 2. 
\end{corollary}

\subsection{From graphs on surfaces to Riemannian surfaces}

A \emph{Riemannian surface} $(S,m)$ is a surface $S$ (which is assumed
to be compact, as defined above) together with  a
metric $m$, defined by a scalar product on the tangent space of every
point. The only property that we will need is that for any point $p\in
S$, there is a small open neighbourhood $N$ containing $p$ that is strongly
convex, i.e.\ any two points in $N$ are joined by a unique shortest
path. For  more background on Riemannian surfaces, the
interested reader is referred to the standard textbook~\cite{Spi99}.

\medskip

The following result appears to be well known in the area. For
instance it can be deduced from the work of Saucan~\cite{Sau09}. Here we
include a simple proof (suggested to us by Ga\"el Meignez) in
dimension 2.

\begin{lemma}\label{lem:rietogr}
Let $(S,m)$ be a Riemannian surface. Then there is a $(0,1]$-weighted graph $G$, embedded in $S$, such that
any point of $S$ is at distance at most 2 from a vertex of $G$ in $S$ and
for any vertices $x,y \in V(G)$, $d_m(x,y) \le d_G(x,y) \le
5d_m(x,y)+3$. 
\end{lemma}

\begin{proof}
We consider an inclusion-wise maximal set $P$ of points of $S$ that
are pairwise at distance at least $\tfrac15$ apart in $S$. By maximality of $P$, the open balls of radius $\tfrac25$
centered in $P$ cover $S$. Let $G'$ be the graph with vertex set $P$,
in which two points $p,q\in P$ are adjacent if their balls of radius
$\tfrac12$ intersect. For each such pair $p,q$, we join $p$ to $q$ by a
shortest path (of length $d_m(p,q)\le 1$) on the surface $S$. Note that any
two such shortest paths intersect in a finite number of points and
segments (since otherwise we could find an arbitrarily small
neighbourhood containing two points joined by two distinct shortest
paths). For each intersection point between two paths (and each end of
an intersecting segment between the paths), we add a new vertex to
$G'$. Let $G$ be the resulting graph (where two vertices of $G$ are adjacent
if they are consecutive on some shortest path between vertices of $G'$). By definition, $G$ is properly embedded in $S$. Note that each edge $e$ of
$G$ corresponds to a shortest path between the two endpoints of $e$ in $S$ (we denote the length of this shortest path by $\ell_e$). From now on, we
consider $G$ as a weighted graph with weights $(\ell_e)_{e\in E(G)}$,
and all path lengths and distances in $G$ refer to the weighted shortest path metric
induced by $(\ell_e)_{e\in E(G)}$. Note that by definition, all the weights are in the interval $(0,1]$. For any two vertices
$p,q$ in $G$, we clearly have $d_m(p,q)\le d_G(p,q)$. Consider now a
length-minimizing geodesic $\gamma$ (of length $\ell=d_m(p,q)$) between $p$ and
$q$ in $S$, and take $k\le 5\ell+2$ points $p_1,\ldots,p_k$ (in this
order) on
$\gamma$, with $p_1=p$, $p_k=q$, and such that any two consecutive
points $p_i,p_{i+1}$ are at distance at most $\tfrac15$ apart in $S$. Recall
that each point of $S$ is at distance at most $\tfrac25$ from a point of $P$. It follows that any
segment $[p_i,p_{i+1}]$ of $\gamma$ is contained in a ball of radius
$\tfrac1{10}+\tfrac25=\tfrac12$ centered in a point $r_i$ of $P$ (which is a vertex of
$G'$). By definition of $G'$, for any $1\le i\le k-1$, $r_i$ and
$r_{i+1}$ coincide or are adjacent in $G'$. Note
that $p=p_1$ is on some edge of length at most $2\cdot \tfrac12=1$
between two vertices of $V(G')=P$, and thus $p$ is at distance at most
$\tfrac12$ from one of these points (call it $r_0$) in $G$. In
particular, by definition of $G'$, $r_0$ coincides or is adjacent with
$r_1$ in $G'$. Similarly, $q=p_k$ is at distance at most $\tfrac12$ in
$G$ from a vertex $r_k\in V(G')=P$ that coincides or is adjacent to
$r_{k-1}$ in $G'$. 
As a consequence, there is
a path of length at most $2\cdot \tfrac12 +2\cdot \tfrac12\cdot k \le
5\ell+3 $ between  $p$ and $q$ in $G$.
\end{proof}

\subsection{Quasi-isometry}

Two metric spaces $(X,d_X)$ and $(Y,d_Y)$ are \emph{quasi-isometric}
if there is a map $f: X \rightarrow Y$ and constants
$\epsilon\ge 0$, $\lambda\ge 1$, and $C\ge 0$ such that $Y \subseteq N_C(f(X))$
(recall that $N_C(S)$ denotes the $C$-neighbourhood of $S$, so this condition means that for any $y\in Y$ there
is $x\in X$ such that $d_Y(y,f(x))\le C$), and for every $x_1,x_2\in
X$, $$\frac1{\lambda}d_X(x_1,x_2)-\epsilon\le d_Y(f(x_1),f(x_2))\le
\lambda d_X(x_1,x_2)+\epsilon.$$

It is not difficult to check that the definition is symmetric.  Moreover,
if for every $r\ge 0$, $X$ has a cover by $n$ sets whose
$r$-components are $D_X(r)$-bounded  and
there exists a map $f: X \rightarrow Y$ as above, then for every $r\ge
0$, $Y$ has a cover by $n$-sets whose $r$-components are
$D_Y(r)$-bounded, where $D_Y$ only depends on $D_X$ and the constants $\lambda$,
$\epsilon$, and $C$ in the definition of $f$. Moreover,  $D_X$ is
linear if and only if $D_Y$ is linear.
This implies
that asymptotic dimension (of linear type) is invariant under quasi-isometry. Moreover,
if all members of a family $\mathcal{X}=X_1,X_2,\ldots$ of metric spaces are
quasi-isometric to some metric space $Y$, with uniformly bounded
constants $\lambda$,
$\epsilon$, and $C$ in the definition of the quasi-isometry map, then
$\mathrm{asdim}\, \mathcal{X}\le\mathrm{asdim}\,Y$, and the same holds
for the asymptotic dimension of linear type.

\medskip

Combining Corollary~\ref{cor:genus2} with Lemma~\ref{lem:rietogr} and
the remarks above on the invariance of asymptotic dimension of linear
type under
quasi-isometry, we obtain the following result. 

\begin{thm}\label{thm:riegenus2}
For any integer $g\ge 0$, the class of compact Riemannian surfaces of Euler genus $g$ has
asymptotic dimension at most 2 of linear type.
\end{thm}

\begin{proof}
By
Lemma~\ref{lem:rietogr}, every compact Riemannian surface $(S,m)$ of genus $g$
is quasi-isometric to some $(0,1]$-weighted graph
embeddable on $S$, with constants in the quasi-isometry that are
uniform (in fact the constants are even independent of $g$). By
Corollary~\ref{cor:genus2}, such graphs have asymptotic dimension at
most 2 of linear type (uniformly, i.e.\ with constants that only depend on $g$), and
thus the class of compact Riemannian surfaces of Euler genus $g$ has
asymptotic dimension at most 2 of linear type.
\end{proof}

\section{Topological graphs and \texorpdfstring{$H$}\ -minor free
  graphs of bounded degree}\label{sec:topo}

For a graph $G$, a \emph{tree-decomposition} of $G$ consists of a tree
$T$ and collection $\mathcal{B}=(B_x : x\in V(T))$ of subsets of $V(G)$, called \emph{bags}, indexed by the nodes of $T$, such that:
\begin{itemize}
\item  for every vertex $v$ of $G$, the set $\{x\in V(T) : v\in B_x\}$ induces a non-empty subtree of $T$, and 
\item for every edge $vw$ of $G$, there is a vertex $x\in V(H)$ for which $v,w\in B_x$. 
\end{itemize}
The \emph{width} of such a tree-decomposition is $\max\{|B_x|:x\in V(T)\}-1$. 
The \emph{treewidth} of a graph $G$ is the minimum width of a
tree-decomposition of $G$.

\medskip

A graph $G$ has \emph{layered treewidth} at most $t$ if it has a
tree-decomposition and a layering $L_0,L_1,\ldots$ such that each bag
of the tree-decomposition intersects each layer in at most $t$
vertices. Note that if such a layering exists, then any $\ell$
consecutive layers induce a graph of treewidth at most $t\ell$ (it
suffices to restrict the tree-decomposition of $G$ to the $\ell$ layers under
consideration, and thus each bag contains at most $t\ell$
vertices). It follows that families of graphs of bounded layered
treewidth are layerable by families of bounded treewidth.

\smallskip

Using the terminology of layerable classes, these results can be rephrased
as follows.

\begin{thm}\label{thm:ltw}
Let $\mathcal{T}=(\mathcal{T}_i)_{i\in \mathbb{N}}$, where
$\mathcal{T}_i$ is the class of graphs of treewidth at most $i$. Then
for any $t$, the class of graphs of layered treewidth at most $t$ is
$t$-linearly $\mathcal{T}$-layerable.
\end{thm}

Theorems \ref{thm:bd}
and \ref{thm:ltw} thus have the following immediate consequence.

\begin{thm}\label{thm:ltwasdim}
  Assume that there exists $k$, such that for any $t$, the class of
  graphs of treewidth at most $t$ has asymptotic dimension at most
  $k$. Then any class of bounded layered treewidth has asymptotic
  dimension at most $t+1$.
\end{thm}

A graph is \emph{$k$-planar} if it can be drawn in the plane with at most $k$ crossings per edge~\cite{PT97}. More generally, a graph is \emph{$(g,k)$-planar} if it can be drawn in a surface of Euler
genus $g$ with at most $k$ crossings per edge.
It was proved by Dujmovi\'c, Eppstein, and Wood~\cite{DEW17} that for any $g$ and $k$ the class of $(g,k)$-planar
graphs has bounded layered treewidth. An \emph{apex graph} is a graph $G$ such that there
is a vertex $v$ with the property that $G-v$ is planar. It is well
known that excluding a
planar graph as a minor is equivalent to having bounded
treewidth. Dujmovi\'c, Morin, and Wood~\cite{DMW17}  
proved that for any apex graph $H$, any graph excluding $H$ as a minor
has a bounded layered treewidth (they also prove the converse for
minor-closed classes). We thus obtain the following immediate
corollary of Theorem~\ref{thm:ltwasdim}.

\begin{corollary}\label{cor:gkplanar}
  Assume that there exists $k$, such that for any $t$, the class of
  graphs of treewidth at most $t$ has asymptotic dimension at most  $k$. Then
\begin{itemize}
\item for any fixed integers $\ell$ and $g$, the class of  $(g,\ell)$-planar
  graphs, and
  \item for any apex graph $H$, the class of
    $H$-minor free graphs
    \end{itemize}
    have  asymptotic dimension at most $k+1$. 
  \end{corollary}

  Recall that Fujiwara and Papasoglu asked whether there is a constant
$k$ (possibly $k=2$) such that for any proper minor-closed class of
graphs has asymptotic dimension at most $k$ (Question~\ref{qn:fp}). What
Corollary~\ref{cor:gkplanar} shows is that in order to prove that this
holds for graphs excuding an apex-minor, it is enough to prove it for
classes of
graphs of bounded treewidth.

We now prove that the question of Fujiwara and Papasoglu has a
positive answer when we restrict ourselves to graphs of bounded
degree. Recall that since any infinite
family of cubic expanders has unbounded asymptotic dimension~\cite{Hum17}, restricting the maximum degree only is not
enough to imply that a class of graphs has bounded asymptotic dimension.

\medskip

  In~\cite{DO95} (see also~\cite{Woo09}), it is proved that if a graph $G$ has treewidth $t\ge 1$ and
maximum degree $\Delta\ge 1$, then $G$ is a subgraph of the strong
product $T\boxtimes K_{24t\Delta}$, where $T$ is a tree (an equivalent
formulation is that $G$ can be obtained as a subgraph of a tree $T$ in
which each vertex has been replaced with a clique on $24t\Delta$
vertices and each edge by a complete bipartite graph between the
corresponding cliques). This implies the following layerability
result.

\begin{thm}\label{thm:twdeglayer}
Let $\mathcal{S}=(\mathcal{S}_i)_{i\in \mathbb{N}}$, where
$\mathcal{S}_i$ is the class of graphs whose connected components have
size at most $i$. Then for any integers $t$ and $\Delta$, the class of
graphs of treewidth at most $t$ and maximum degree at most $\Delta$ is $\mathcal{S}$-layerable.
\end{thm}

\begin{proof}
Let $G$ be a graph of treewidth at most $t$ and maximum degree at most
$\Delta$, and let $T$ be a tree such that $G$ is a subgraph of
$T\boxtimes K_{24t\Delta}$. Note that since $G$ has maximum degree
$\Delta$, and is a subgraph of $T\boxtimes K_{24t\Delta}$, we can
assume that $T$ has maximum degree at most $24t\Delta \cdot \Delta=24t\Delta^2$.

We consider a BFS-layering of $T$ (see Section~\ref{sec:layer}), and the
induced layering of $G$ (i.e. a vertex of $G$ is in layer $i$ if the
corresponding vertex of $T$ is in layer $i$ of the BFS-layering of
$T$). Call $L_0,L_1,\ldots$ this layering of $G$ (it can be checked that it is indeed a layering, since any edge of $G$ is either
inside some $K_{24t\Delta}$ or between two such cliques). Consider
$\ell$ consecutive layers in this layering. In $T$, the corresponding
layers are also consecutive and since $T$ is a tree of maximum degree
at most $24t\Delta^2$, each component of the subgraph of $T$ induced
by these $\ell$ layers is a tree of at most $(24t\Delta^2)^\ell$
vertices. It follows that any  $\ell$ consecutive layers in $G$ induce
a graph whose connected components have size at most
$24t\Delta(24t\Delta^2)^\ell$. This implies that the class of
graphs of treewidth at most $t$ and maximum degree at most $\Delta$ is $\mathcal{S}$-layerable.
\end{proof}

Since for any $i$, the class $\mathcal{S}_i$  of graphs whose connected components have
size at most $i$ has asymptotic dimension 0, we obtain the following
as a direct consequence of Theorems~\ref{thm:bd} and~\ref{thm:twdeglayer}.

\begin{thm}\label{thm:twdeg}
For any integers $t$ and $\Delta$, the class of
graphs of treewidth at most $t$ and maximum degree at most $\Delta$
has asymptotic dimension at most 1.
\end{thm}

As pointed out to us by David Hume, Theorem~\ref{thm:twdeg} can also
be deduced from the work of Benjamini, Schramm, and Tim\'ar~\cite{BST10}.
Theorems~\ref{thm:twdeg} and~\ref{thm:ltwasdim} now imply the
following.

\begin{thm}\label{thm:ltwdeg}
For any integers $t$ and $\Delta$, the class of
graphs of layered treewidth at most $t$ and maximum degree at most $\Delta$
has asymptotic dimension at most 2.
\end{thm}

In~\cite{DEMWW}, it is proved that graphs of bounded
degree from a proper minor-closed class have bounded layered treewidth (this is a direct
consequence of Theorem 19 and Lemma 12 in the paper). We thus obtain Theorem~\ref{thm:minordeg} as a direct consequence of Theorem~\ref{thm:ltwdeg}. Note that the proof of Theorem 19 in~\cite{DEMWW} uses the graph minor structure theorem by Robertson and Seymour~\cite{RS03}, and therefore Theorem~\ref{thm:minordeg} is also based on this deep result.

Using that $(g,k)$-planar graphs have bounded layered treewidth in combination
with Theorem~\ref{thm:ltwdeg}, we also obtain the following
unconditional counterpart of Corollary~\ref{cor:gkplanar}.

\begin{thm}
For any integers $g\ge 0$, $k\ge 0$ and $\Delta\ge 1$, the class of
$(g,k)$-planar graphs of maximum degree at most $\Delta$ has asymptotic
dimension at most 2.
\end{thm}

The hypothesis of Corollary~\ref{cor:gkplanar} is that there exists $k$, such that for any $t$, the class of
  graphs of treewidth at most $t$ has asymptotic dimension at most
  $k$. A natural question is whether this might hold not only for the
  asymptotic dimension, but also for the
  asymptotic dimension of linear type, or equivalently (for unweighted
  graphs) for the Assouad-Nagata dimension.

  \smallskip

  We now prove that the answer to this stronger question is
  negative for 1-planar graphs, and thus for graphs of bounded layered
  treewidth.

  \medskip

  Given a graph $G$ and an integer $k$, the \emph{$k$-subdivision} of
  $G$, denoted by $G^{(k)}$, is obtained from $G$ by replacing each
  edge of $G$ by a path on $k+1$ edges.
  We start with the following simple observation.

  \begin{obs}\label{obs:subdv}
    For every integers $k$ and $n$, if $G^{(k)}$ admits an $n$-dimensional control
    function $D$ which is a dilation, then $D$ is also an $n$-dimensional control
    function for $G$.
\end{obs}

\begin{proof}
  Let $D(r)= c\cdot r$, for some constant $c>0$.
  Consider the graph $G$ and fix some $r>0$. Let $U_1,\ldots,U_{n+1}$
  be a cover of $G^{(k)}$ by $n+1$ sets whose $(k+1)r$-components are
  $D((k+1)r)$-bounded, and consider the restriction $U_1',\ldots,U_{n+1}'$ of this cover to
  $V(G)$.  In $(G,d_G)$, each $r$-component of $U_i'$ (for some $1\le
  i \le n+1$) is precisely the restriction of a $(k+1)r$-component of
  $U_i$ (in $G^{(k)}$ and the associated graph metric) to $V(G)$, and
  thus each $r$-component of $U_i'$ is $D((k+1)r)$-bounded with
  respect to the shortest path metric in $G^{(k)}$. Note that
  $D((k+1)r)= c\cdot (k+1)r$  and for any two vertices $u$ and $v$
  in $G$, we have $d_{G^{(k)}}(u,v)=(k+1)\cdot d_G(u,v)$. It follows
  that for any $1\le i \le n+1$, each $r$-component of $U_i'$ is
  $cr$-bounded, and thus $D(r)$-bounded, as desired.
\end{proof}

Now, take any family $\mathcal{F}$ of graphs of unbounded
Assouad-Nagata dimension (for instance, grids of increasing size and dimension), and for
each graph $G \in \mathcal{F}$, let $G^*:=G^{(|E(G)|)}$ (i.e. subdivide
each of $G$ as many times as the number of edges of $G$). Note that
$G^*$ is 1-planar (this can be seen by placing the vertices of $G$ in general position in the plane, joining adjacent vertices in $G$ by straight-line segments, and then subdividing each edge at least once between any two consecutive
crossings involving this edge). Define $\mathcal{F}^*:=\{G^*\,|\,G
\in \mathcal{F}\}$, and note that all the graphs of  $\mathcal{F}^*$
are 1-planar. If the class of 1-planar graphs had bounded
Assouad-Nagata dimension, then
by Observation~\ref{obs:subdv}, $\mathcal{F}$ would have
bounded Assouad-Nagata dimension, which is a contradiction.
We obtain the following.

\begin{thm}\label{thm:nofptw}
There is no integer $k$ such that  the class of 1-planar graphs has Assouad-Nagata dimension at most $k$.
\end{thm}

As we have seen in the introduction, Assouad-Nagata dimension and
asymptotic dimension of linear type are equivalent in unweighted
graphs, so the analogous result holds for the asymptotic dimension of
linear type for unweighted graphs.

\section{Geometric graphs and graphs of polynomial growth}\label{sec:geom}

We now explain some consequences of Theorem~\ref{thm:bd} for
geometric graph classes and graph classes of polynomial growth.

\medskip

For an integer $d\ge 1$, and some real $C\ge 1$, let $\mathcal{D}^d(C)$ be the class
of graphs $G$ whose vertices can be mapped to points of $\mathbb{R}^d$
such that
\begin{itemize}
\item any two vertices of $G$ are mapped to points at (Euclidean) distance at
  least 1 apart, and
\item any two adjacent vertices of $G$ are mapped to points at distance at
  most $C$ apart.
\end{itemize}

\smallskip

For technical reasons it will be convenient to consider also a
sequence of subclasses of $\mathcal{D}^d(C)$. For $1\le i\le d$, and some
real $K>0$, we let $\mathcal{D}^d_i(C,K)$ be the set of graphs of
$\mathcal{D}^d(C)$ such that the points in the definition are restricted
to the subspace $\mathbb{R}^i\times[0,K]^{d-i}$ of $\mathbb{R}^d$. Note that the box
$[0,K]^{d-i}$ can be replaced with any translation of  $[0,K]^{d-i}$ without changing
the definition of $\mathcal{D}^d_i(C,K)$. Note also that
$\mathcal{D}^d_0(C,K)\subseteq \mathcal{D}^d_1(C,K) \subseteq \cdots
\subseteq \mathcal{D}^d_d(C,K)=\mathcal{D}^d(C)$.

\smallskip

Given a graph $G\in \mathcal{D}^d(C)$ and a mapping of its vertices in
$\mathbb{R}^d$ as in the definition, a natural layering of $G$ can be
obtained as follows: fix some dimension $1\le j \le d$, and for any
$k\in \mathbb{Z}$, let $L_k$ be the set of vertices of $G$ whose
 $j$-th coordinate lie in the interval $[kC,(k+1)C)$. By definition of
 $\mathcal{D}^d(C)$ this is clearly a layering, and any $\ell$
 consecutive layers induce a graph of
 $\mathcal{D}_{d-1}^d(C,\ell C)$. More generally, if $i\ge 1$ and $G\in
 \mathcal{D}^d_i(C,K)$, and the dimension chosen for the definition of
 the
 layering of $G$ is $j\le i$, then any $\ell$ consecutive layers in the
 layering induce a graph of $\mathcal{D}_{i-1}^d(C,\max(K,\ell C))$.

 \medskip

 With these observations in mind, the following result is now a fairly direct
 consequence of Theorem~\ref{thm:bd}.

 \begin{thm}\label{thm:krlee}
For any integer $d\ge 1$ and real $C\ge 1$, the class $\mathcal{D}^d(C)$ has asymptotic
dimension at most $d$.
 \end{thm}

 \begin{proof}
We will prove by induction on $0\le i \le d$ that for any fixed $K$,
the class $\mathcal{D}^d_i(C,K)$ has asymptotic dimension at most
$i$. Since  for every real $K>0$,
$\mathcal{D}^d_d(C,K)=\mathcal{D}^d(C)$, this will prove the desired result.

We start with the base case $i=0$. Then for any representation of a
graph $G \in \mathcal{D}^d_0(C,K)$ in $\mathbb{R}^d$, all the vertices
of $G$ are mapped to $[0,K]^d$. Since the points are pairwise at distance at
least 1 apart, a simple volume computation shows that $G$ contains a
bounded number of vertices, and thus $\mathcal{D}^d_0(C,K)$ has
asymptotic dimension $0$.

Assume now that $1\le i\le 1$, and for any fixed $K'$, the class $\mathcal{D}^d_{i-1}(C,K')$ has
asymptotic dimension at most $i-1$. Consider a real $K>0$ and a
graph $G \in \mathcal{D}^d_{i}(C,K)$, together with a representation of
$G$ in $\mathbb{R}^i\times[0,K]^{d-i}$. Construct the natural layering
of $G$ (defined above) in dimension $i$. As observed above, any $\ell$
layers in this layering induce a graph of
$\mathcal{D}_{i-1}^d(C,\max(K,\ell C))$. By defining the sequence
$\mathcal{B}=(\mathcal{D}_{i-1}^d(C,\max(K,\ell C)))_{\ell\in
  \mathbb{N}}$, we obtain that the class $\mathcal{D}^d_{i}(C,K)$ is
$\mathcal{B}$-layerable. By the induction hypothesis, every class of $\mathcal{B}$ has
asymptotic dimension at most $i-1$, thus it follows from
Theorem~\ref{thm:bd} that $\mathcal{D}^d_{i}(C,K)$ has asymptotic
dimension at most $i-1+1=i$, as desired.
 \end{proof}
 
 We now explore some consequences of Theorem~\ref{thm:krlee}.

 \smallskip

Consider first the class of \emph{subgraphs} of contact graphs of unit balls in
$\mathbb{R}^d$. A graph $G$ is in this class if the vertices of $G$
can be mapped to points of $\mathbb{R}^d$ such that any two vertices
are mapped to points at distance at least 1 apart, while adjacent
vertices are mapped to points at distance exactly 1 apart. This class
is precisely $\mathcal{D}^d(1)$, and thus Theorem~\ref{thm:krlee}
implies that it has asymptotic dimension at most $d$ (note that
this class is a natural generalization of $d$-dimensional grids, whose
asymptotic dimension is $d$).

\begin{corollary}\label{cor:contactunitball}
For any $d\ge 1$, the class of subgraphs of contact graphs of unit balls in
$\mathbb{R}^d$ has asymptotic dimension at most $d$.
\end{corollary}

Note that since $d$-dimensional grids are contact graphs of unit balls in $\mathbb{R}^d$, this immediately implies Theorem~\ref{thm:subgrid}.

\medskip

We note that the property that pairs of vertices are mapped to points
at distance at least 1 apart in the definition of $\mathcal{D}^d(C)$ is only used in the volume argument
showing that any bounded region contains a bounded number of points. So
we could replace it 
by the condition that any unit $d$-cube (the cartesian product of $d$
unit intervals) in
$\mathbb{R}^d$ contains at most $k$ points, for some
universal constant $k$, without changing the results.

\smallskip

Note that all the unconditional results that have been
obtained in this section and the previous one concern graphs of
bounded degree, which in particular are metric spaces of \emph{bounded
  geometry} (i.e.\ metric spaces for which there is a function $f$ such
any $r$-ball contains at most $f(r)$ points).
We now give a  consequence of Theorem~\ref{thm:bd} for a graph
class with unbounded degree. A \emph{unit ball graph} in
$\mathbb{R}^d$ is the intersection graph of a set of unit balls in
$\mathbb{R}^d$. In other words, the vertices of a unit ball graph $G$
can be mapped to points of $\mathbb{R}^d$ such that any two vertices
are adjacent if and only if the corresponding points are at
(Euclidean) distance
at most 1 apart. It should be mentioned that this is very different from
the previous application on (subgraphs of) contacts of unit
balls. Here we really need to consider the intersection graph itself, not
some of its subgraphs.

\begin{thm}\label{thm:unitball}
For any integer $d$, the class of unit ball graphs in $\mathbb{R}^d$
has asymptotic dimension at most $d$.
\end{thm}

\begin{proof}
  We denote the class of unit ball graphs in $\mathbb{R}^d$ by
  $\mathcal{B}^d$. As in the proof of Theorem~\ref{thm:krlee}, for any
  real $K>0$ and $0\le i \le d$ we consider the class $\mathcal{B}^d_i(K)$ of unit ball graphs whose points are restricted to
  the subspace $\mathbb{R}^{i}\times [0,K]^{d-i}$. We again note that
  $[0,K]^{d-i}$ can be replaced with the cartesian product of $d-i$
  intervals of length $K$ without changing the definition of the
  class, and $\mathcal{B}^d_0(K)\subseteq \mathcal{B}^d_1(K) \subseteq
  \cdots \subseteq \mathcal{B}^d_d(K)=\mathcal{B}^d$.

  We will prove by induction on $0\le i \le d$ that for any fixed $K$,
the class $\mathcal{B}^d_i(K)$ has asymptotic dimension at most
$i$. Since  for every real $K>0$,
$\mathcal{B}^d_d(K)=\mathcal{B}^d$, this will prove the desired result.

We start with the base case $i=0$. Then for any representation of a
graph $G \in \mathcal{B}^d_0(K)$ in $\mathbb{R}^d$, all the vertices
of $G$ are mapped to $[0,K]^d$. We claim that each connected component
of $G$ has diameter at most $2(\sqrt{d}K)^d$. To see this, take a shortest path $P$ of
length more than $2(\sqrt{d}K)^d$ in $G$. Observe that $[0,K]^d$ can be partitioned into $(\sqrt{d}K)^d$ cubes
of diameter at most 1, and that by the pigeonhole principle some of
these cubes contains at least 3 vertices of $P$. Since this cube has
diameter at most 1, the three corresponding vertices are pairwise
adjacent in $G$,
which contradicts the fact that $P$ is a shortest path (since a
shortest path is an induced path). Hence, all the connected components
of $\mathcal{B}^d_0(K)$ are
$2(\sqrt{d}K)^d$-bounded, and thus it follows that
$\mathcal{B}^d_0(K)$ has
asymptotic dimension $0$.

The rest of the proof is now exactly the same as that of Theorem~\ref{thm:krlee}.
Assume that $1\le i\le 1$, and for any fixed $K'$, the class $\mathcal{B}^d_{i-1}(K')$ has
asymptotic dimension at most $i-1$. Consider a real $K>0$ and a
graph $G \in \mathcal{B}^d_{i}(K)$, together with a representation of
$G$ in $\mathbb{R}^i\times[0,K]^{d-i}$. Construct the natural layering
of $G$ (defined before Theorem~\ref{thm:krlee}) in dimension $i$. As before, any $\ell$
layers in this layering induce a graph of
$\mathcal{B}_{i-1}^d(\max(K,\ell))$. By defining the sequence
$\mathcal{B}=(\mathcal{B}_{i-1}^d(\max(K,\ell)))_{\ell\in
  \mathbb{N}}$, we obtain that the class $\mathcal{B}^d_{i}(K)$ is
$\mathcal{B}$-layerable. By the induction hypothesis, every class of $\mathcal{B}$ has
asymptotic dimension at most $i-1$, thus it follows from
Theorem~\ref{thm:bd} that $\mathcal{B}^d_{i}(K)$ has asymptotic
dimension at most $i-1+1=i$, as desired.
\end{proof}

As before, we have restricted ourselves to unit balls, but the only
property that has been used is that the ratio between the largest
radius and the smallest radius of a ball is bounded uniformly. It follows that our results hold in this wider setting as well.

\medskip

A graph has \emph{growth} at most $f$ if every $r$-ball contains at most
$f(r)$ vertices. A graph has \emph{growth rate} at most $d$ if it has
growth at most $f(r)=r^d$ (for $r>1$). Classes of graphs whose growth
rate is uniformly bounded are also known as \emph{classes of
  polynomial growth}.
In~\cite{KL03} it is proved that
any graph $G$ with growth rate at most $d$ is in the class
$\mathcal{D}^{O(d\log d)}(2)$.
We thus obtain the following result as a direct corollary of
Theorem~\ref{thm:krlee}.

\begin{corollary}\label{cor:polygrowth}
For any real $d> 1$, the class of
graphs of growth rate at most $d$ has asymptotic dimension
$O(d \log d)$. In particular, classes of graphs of polynomial growth
have bounded asymptotic dimension.
\end{corollary}

Again, this is a fairly natural extension of the fact that
$d$-dimensional grids have bounded asymptotic dimension
($d$-dimensional grids form the basic example of graphs of polynomial
growth).
It seems that previously, it was only known that
\emph{vertex-transitive} graphs of polynomial growth have bounded
asymptotic dimension~\cite{Hum17}.

\medskip

Let us now argue why the assumption that the growth is polynomial in
Corollary~\ref{cor:polygrowth} cannot be weakened.
We say that a function $f$ is \emph{superpolynomial} if it can be written as
$f(r)=r^{g(r)}$ with $g(r)\to \infty$ when $r\to \infty$.

\begin{thm}\label{th:superpoly:nope}
For any superpolynomial function $f$ that satisfies $f(r) \geq 3r+1$ for every $r \in \mathbb{N}$, the class of graphs with growth at most $f$ has unbounded asymptotic dimension.
\end{thm}

\begin{proof}
Given $p\ge 0$, a \emph{$p$-subdivided 3-regular tree} is obtained from a tree in which all internal vertices have degree 3 by subdividing each edge $p$ times.
Given two integers
$k,p$ and a graph $G$, we say that a graph $G'$ is a \emph{$(k,p)$-stretch} of $G$ if it is obtained by subdividing every edge of $G$ $k$ times, and replacing every vertex $v\in V(G)$ with a
$p$-subdivided 3-regular tree $T_v$ with $d_{G}(v)$ leaves, and with minimal radius with respect to this conditions, in an arbitrary way (see
Figure~\ref{fig:grid} for an illustration). For any set $S$ of vertices of $G'$, the \emph{projection} of $S$ in $G$ is the set of vertices $v$ of $G$ such that $T_v\cap S\ne \emptyset$.   

For every $d \geq 2$, let $\mathcal{G}_d$ be the class of
$d$-dimensional grids. Recall from Section~\ref{sec:intro} that
$\mathcal{G}_d$ has asymptotic dimension $d$. We now define $\mathcal{G}^{k,p}_d$ to be the class of all $(k,p)$-stretches of graphs from $\mathcal{G}_d$. Note that all vertices $v$ of a graph of $\mathcal{G}_d$ have degree at most $2d$, and thus each tree $T_v$ as above has radius at most $\lceil \log_2(2d)\rceil\le 4\log_2 d$ and diameter at most $2\lceil \log_2(2d)\rceil\le 8 \log_2 d$.

We first show that
$\mathcal{G}^{k,p}_d$ has asymptotic dimension at least $d$, and then
that the growth of $\mathcal{G}^{k,p}_d$ can be controlled by tuning the parameters $k$ and $p$.

\begin{figure}[htb]
 \centering
 \includegraphics[scale=1]{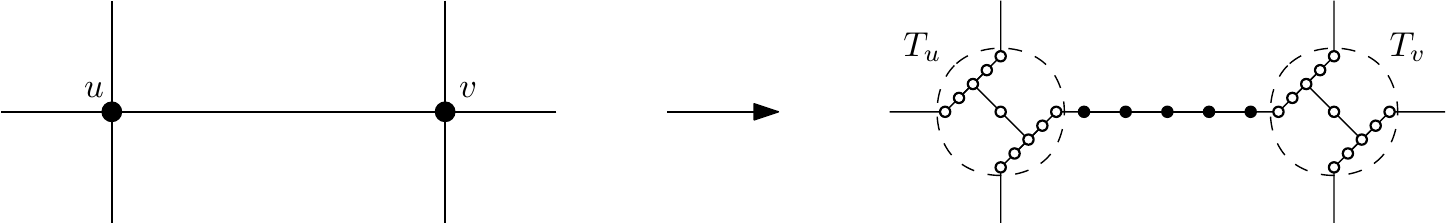}
 \caption{A local view of the $(k,p)$-stretch of a graph, for $p=1$ and $k=5$.}
 \label{fig:grid}
\end{figure}

\begin{claim}
For every $d,k,p$, the class $\mathcal{G}^{k,p}_d$ has asymptotic
dimension at least $d$.
\end{claim}

Assume for a contradiction that there is a function $D'$ such that for every $r \in \mathbb{N}$, every $G' \in \mathcal{G}^{k,p}_d$ can be covered by $d$ sets whose $r$-components are $D'(r)$-bounded. 

Consider a graph $G\in \mathcal{G}_d$, and a $(k,p)$-stretch $G' \in
\mathcal{G}^{k,p}_d$ of $G$. By assumption, there are $d$ sets $U_1',\ldots, U_d'$
that cover $G'$ in such a way that for each $1\le i \le d$, each $(k+16p \log_2 d)r$-component of $U_i'$ is
$D'((k+16p \log_2 d)r)$-bounded. For $1\le i \le d$, let $U_i$ be the projection of $U_i'$ in $G$. Note that $U_1,
\ldots, U_d$ forms a cover of $G$, and every $r$-component of some $U_i$ is a
subset of the projection of a $(k+16p \log_2 d)r$-component  
of $U_i'$ in $G$. 
Therefore, for every
$i$, every $r$-component of $U_i$ is $D'((k+16p \log_2 d)r)$-bounded. By
setting $D(r)=D'((k+16p \log_2 d)r)$, this shows that $D$ is a
$(d-1)$-dimensional control function for $\mathcal{G}_d$, a
contradiction.\hfill $\diamond$

\begin{claim}
For every $d,k,p \in \mathbb{N}$ with $p \geq d$ and $k \geq 4p \log_2 d$, the class $\mathcal{G}^{k,p}_d$ has growth at most $r\mapsto 3r+1$ when $r \leq p$, $r \mapsto 4dr+1$ when $p \leq r \leq k$, and $r \mapsto 4dk  \cdot(1+2\lceil r/k\rceil)^d$ when $r \geq k$.
\end{claim}

Given a graph $G \in \mathcal{G}^{k,p}_d$ and a vertex $u \in G$, we bound the number $|B_r(u)|$ of vertices at distance at most $r$ of $u$ in $G$, as a function of $r$. Observe that if $B_r(u)$ induces a tree with at most $t$ leaves, then $|B_r(u)|\le tr+1$.

If $r \leq p$, then $B_r(u)$ is a tree with at most 3 leaves, and thus $|B_r(u)|\le3 r+1$. If $p\le r \le k$, then $B_r(u)$ is a tree with at most $4d$ leaves (corresponding to the leaves of two adjacent trees $T_v$ and $T_w$, each having at most $2d$ leaves), and thus $|B_r(u)|\le 4dr+1$. Assume now that $r\ge k$. A graph in $\mathcal{G}_d$ has growth rate at most $r \mapsto  (2r+1)^d$. Moreover, a graph of $\mathcal{G}_d$ can be obtained from $G$ by contracting trees of radius at most $\tfrac{k+1}2+4p\log_2 d$ 
with $2d$ leaves (corresponding to each tree $T_v$ together with half of each of the $2d$ incident subdivided edges) into single vertices. By the observation above, each such tree contains at most $2d(\tfrac{k+1}2+4p\log_2 d)+1\le 4dk$ vertices. Therefore, the vertex $u$ is at distance at most $r$ from at most $4dk  \cdot(1+2\lceil r/k\rceil)^d$ 
vertices in $G$.\hfill $\diamond$

\medskip

Consider now a superpolynomial function $f$ with $f(r) \geq 3 r+1$ for every $r$, and fix an integer $d$. Since $f$ is superpolynomial, there exists an integer $p \geq d$ such that $f(r) \geq 4dr+1$ for every $r \geq p$. Similarly, there exists an integer $k \geq 4p \log_2 d$ such that $f(r) \geq 4dk \cdot (1+2\lceil r/k\rceil)^d$ for every $r \geq k$. Therefore it follows from the two claims above that the class $\mathcal{G}^{k,p}_{d}$  has growth at most $f$ and asymptotic dimension at least $d$. Since $d$ was arbitrary, it follows that the class of graphs of growth at most $f$ has unbounded asymptotic dimension.
\end{proof}

We note that the intriguing assumption of Theorem~\ref{th:superpoly:nope}, requiring that $f(r)\ge 3r+1$ for every $r$, turns out to be necessary. If $f(r_0)\leq 3r_0$ for some $r_0$, then the class of graphs with growth at most $f$ has asymptotic dimension $1$. To see this, observe first that graphs with no cycle of length at least $2r_0+1$ have no $(r_0+1)$-fat 2-banana, and thus it follows from Lemma~\ref{lem:dim_theta_free} that they have a 1-dimensional control function. On the other hand, if a connected graph $G$ of growth at most $f$ contains a cycle $C$ of length at least $2r_0+1$, 
then any vertex of $G$ is at distance less than $r_0$ from $C$ (since otherwise some ball of radius $r_0$ centered in a vertex of $C$ would contain at least $3r_0+1$ vertices). In this case it is easy to find a cover by two sets whose $r$-components are $\max(2r+2r_0,4r_0)$-bounded, by dividing $C$ into an even number of intervals of length in $(r,2r]$, colouring the intervals with alternating colours 1 and 2, and adding all the remaining vertices to their closest interval (breaking ties arbitrarily).

\bigskip

It is worth mentioning that the class $\mathcal{G}^{k,p}_d$ defined in the proof of Theorem~\ref{th:superpoly:nope} has interesting properties. The graphs in the class have maximum degree 3 and the vertices of degree 3 are arbitrarily far apart (by taking $p$ arbitrarily large), so in particular it gives a simple example of a graph class of bounded degree and unbounded asymptotic dimension (without relying on expansion properties). It also allows for a precise forbidden (induced) subgraph characterisation of bounded asymptotic dimension. Given a graph $H$, let $\text{Forb}(H)$ be the class of graphs excluding $H$ as a subgraph. 

Assume that every connected component of $H$ is a path or a subdivided $K_{1,3}$, and let us argue by induction on the number of connected components that $\text{Forb}(H)$ has asymptotic dimension at most 1. If $H$ has a single component, then $H$ is itself a path or a subdivided $K_{1,3}$, and an argument similar to the proof above that any class of growth $f$ with $f(r_0)\le 3r_0$ for some $r_0$ has asymptotic dimension 1 shows that $\text{Forb}(H)$ has asymptotic dimension at most 1.
To derive the general case, we note that for any two graph classes $\mathcal{G}$ and $\mathcal{G}'$, if every graph in $\mathcal{G}$ contains a bounded number of vertices whose removal yields a graph in  $\mathcal{G}'$, then the asymptotic dimension of $\mathcal{G}$ is at most that of $\mathcal{G'}$.
It suffices then to observe that if $H$ is not connected and $H'$ is a connected component of $H$, then every graph in $\text{Forb}(H)$ either belongs to $\text{Forb}(H')$ or contains $|V(H')|$ vertices whose removal yields a graph in $\text{Forb}(H \setminus H')$.

On the other hand, if $H$ is a connected graph distinct from a path or a subdivided $K_{1,3}$, then the class $\mathcal{G}^{k,p}_d$ (for increasing $d$ and suitable parameters $k,p$) shows that $\text{Forb}(H)$ has unbounded asymptotic dimension. We thus obtain the following result.

\begin{obs}\label{obs:Hsubgraph}
For any graph $H$, the class $\text{Forb}(H)$ has asymptotic dimension at most 1 if $H$ is a disjoint union of paths and subdivided $K_{1,3}$, and unbounded asymptotic dimension otherwise.
\end{obs}

Given a graph $H$, let $\text{Forb}_i(H)$ denote the class of graphs excluding $H$ as an induced subgraph.
If $H$ is a path (or a disjoint union of paths) then any graph in $\text{Forb}_i(H)$ consists of connected components of bounded diameter, and thus $\text{Forb}_i(H)$ has asymptotic dimension 0. On the other hand, if $H$ has a connected component that is not a path then it is not difficult to check that it does not appear as an induced subgraph in a graph of $\mathcal{G}^{k,p}_d$ for sufficiently large $k$ and $p$ (when $H$ is a subdivided $K_{1,3}$ we need to replace every vertex of degree 3 in graphs of $\mathcal{G}^{k,p}_d$ with a triangle, which does not impact the asymptotic dimension). As a consequence, if $H$ has a connected component other than a path, then $\text{Forb}_i(H)$ has unbounded asymptotic dimension. We thus obtain the following.

\begin{obs}\label{obs:Hindsubgraph}
For any graph $H$, the class $\text{Forb}_i(H)$ has asymptotic dimension $0$ if $H$ is a disjoint union of paths, and unbounded asymptotic dimension otherwise.
\end{obs}

\section{Pathwidth}\label{sec:pw}

The \emph{pathwidth}  of a graph $G$ is the minimum width of a
tree-decomposition $(T,\mathcal{B})$ of $G$ of width at most $k$ such that $T$ is
a path. Equivalently $G$ has pathwidth at most $k$ if and only if $G$
is a spanning subgraph of some interval graph of clique number at most $k+1$.

In this section it will be convenient to use the definition of
asymptotic dimension based on covers of bounded $r$-multiplicity (see
Section~\ref{sec:sparsecover}).

\begin{thm}\label{thm:pw}
  For every $k\ge 1$, and every integer $r\ge 1$, every graph of
  pathwidth at most $k$ has a $O((kr)^{2k+2})$-bounded cover of
  $r$-multiplicity at most 2. In particular, for every $k\ge 1$, the
  class of graph of
  pathwidth at most $k$ has asymptotic dimension at most 1.
\end{thm}

We start with some notation and a few preparatory lemmas. In all this section, $G$ is a graph of pathwidth at most $k$. Let $(P,\mathcal{B})$ be a path-decomposition of $G$ of width $k$ which minimizes $\ell=|P|$. Let $P=v_1,\ldots,v_\ell$, and let the corresponding bags of $\mathcal{B}$ be denoted by $X_1,\dots,X_\ell$. By minimality of $\ell$ we have that for any $1\le i\le \ell-1$, $X_i \not\subseteq X_{i+1}$ and   $X_{i+1}
\not\subseteq X_i$. We write $X_I=\cup_{i\in I}X_i$, so $|X_I|\leq (k+1)|I|$.
The path-decomposition defines an interval representation $(I_v)_{v\in
  V(G)}$ of a supergraph of $G$ by letting $I_v$ be the smallest real interval containing $\{i\in [\ell]:v\in X_i\}$. If $u$ and $v$ are adjacent in $G$ then $I_v$ and $I_u$ intersect, and each point of the real line is contained in at most $k+1$ intervals $I_v$, $v\in V(G)$. 
Note that by minimality of $\ell$, for any $i\in [\ell]$ there is a
vertex $v$ such that $I_v$ starts at $i$.

\medskip

Given a subset $U \subseteq V(G)$, a vertex $u$ is \emph{maximal} in $U$ if, for every $v \ne u$ in $U$, $I_v$ does not strictly contain $I_u$. A subset $U$ of vertices is \emph{non-nested} if every vertex of $U$ is maximal in $U$.
Let $V_1$ be the set of maximal vertices in $V(G)$, and for any $j \ge
2$, let $V_j$ be the set of maximal vertices in $V(G)-\bigcup_{i<j}
V_i$. Note that since $G$ has pathwidth  at most $k$, all the sets
$V_j$ with $j>k+1$ are empty (since otherwise we would find a vertex
$v$ such that $I_v$ is included in at least $k+1$ other intervals
$I_u$, contradicting the assumption that any point is contained in at
most $k+1$ intervals). This way we have partitioned $V(G)$ into at
most $k+1$ sets $V_1,\ldots,V_{p}$, such that $V_j$ is non-nested for
all $j\in[p]$ (see Figure~\ref{fig:pathwidth} for an example).

\medskip

For each integer $j\in [p]$, let $S_j$ be the set of elements $i\in
[\ell]$ such that an interval $I_v$ with $v\in V_{j}$ starts at $i$. Recall
that for any $i\in [\ell]$, some interval starts at $i$, so it follows
that the sets $S_j$ cover $[\ell]$. 
For $j\in[p]$, we define a measure $\mu_j$ on $[\ell]$ by
\[
\mu_j(S)=|S_j\cap S|,
\]
for any subset $S\subseteq [\ell]$. In other words, $\mu_j(S)$ is  the number of elements $i\in S$ such
that at least one vertex of $V_j$ starts at $i$ (see Figure~\ref{fig:pathwidth} for an illustration).

\begin{figure}[htb]
 \centering
 \includegraphics[scale=0.8]{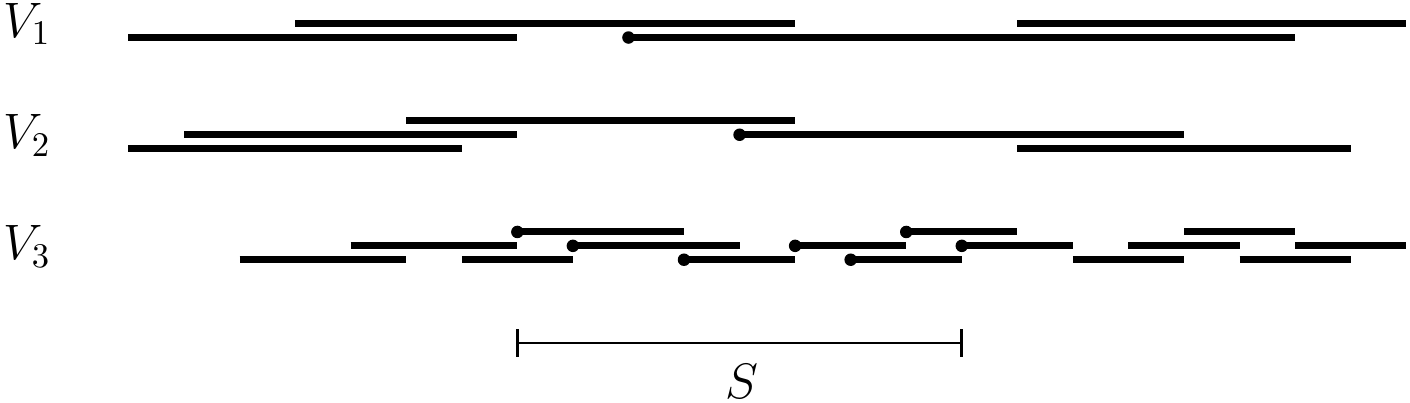}
 \caption{An example of a decomposition into 3 non-nested sets. The interval $S$ has measures $\mu_1(S)=\mu_2(S)=1$ and $\mu_3(S)=7$ (each starting endpoint lying in $S$ is pictured with a dot).}
 \label{fig:pathwidth}
\end{figure}

\begin{lemma}\label{lem:diam}
Let $r,a,b$ be  integers with $[a,b] \subseteq [\ell]$. If $\mu_j([a,b])\leq C$, then any $r$-component of $V_j\cap X_{[a,b]}$ (with respect to the distances in $G$) has diameter at most $(C+1)(k+1)r$ in $G$. 
\end{lemma}

\begin{proof}
We claim that $|X_{[a,b]}\cap V_j|\leq (C+1)(k+1)$. Indeed, $X_a$
contains at most $k+1$ vertices, and intervals corresponding to the
vertices of $V_j$ intersecting $[a,b]$ and not containing $a$ start in at most $C$ elements
of $[a,b]$ (and at most $k+1$ intervals start at each point). The
lemma now follows since any $r$-component of a set of $M$ elements
has diameter at most $M\cdot r$ in $G$.
\end{proof}

We now partition the vertices of each set $V_j$ into
\emph{sections}. We start by defining the following constants:
\[
r_p:=100r, ~R_p:=3(k+1)kr_p^2,\text{ and for any $1\le j<p$,
}r_{j}:=10R_{j+1}\text{ and }R_j:=3(k+1)kr_j^2.
\]
For each $j\in[p]$, we define a partition of $[\ell]$ into consecutive
intervals using $\mu_j$. Each interval $I$ in the partition satisfies
$\mu_i(I)=Q_i:=2kr_i$, apart from the last one that is allowed to have
a smaller weight (these intervals are called \emph{sections of label
  $j$}, or \emph{$j$-sections}).  The set of all $j$-sections is denoted by $P_j$. To each $j$-section, we associate the set of vertices of $V_j$ that \emph{start} in the section. 
Abusing notation, we will often consider a $j$-section  as the set of vertices we associated to it. 

We also number the corresponding sections and hence might refer to the $i$th section of label $j$. With this we then mean those vertices in $V_j$ that started in a bag $a$ with $\mu_j([a])\in [(i-1)\cdot Q_j+1,i\cdot Q_i]$. 
For any $1\le j \le p$, we denote by $G_{\ge j}$ the subgraph of $G$ induced by the vertices $\bigcup_{i \ge j} V_i$.

\begin{lemma}\label{lem:non_consecutive}
Let $A$ and $B$ be two non-consecutive $j$-sections. Then $d_{G_{\ge j}}(A,B) > Q_j/(2k)=r_j$. 
\end{lemma}

\begin{proof}
Let $P=(v_0:=a,\ldots,v_r:=b)$ be a path from a vertex $a\in A$ to a
vertex $b\in B$, that only contains vertices of $G_{\ge i}$. Let $S$
be a $j$-section lying between the $j$-sections $A$ and $B$.
Let us label the elements of $S$ contributing to $\mu_j$ as $x_0,\ldots,x_{Q_i-1}$. 
We prove by induction on $i$ that $I_{v_i}$ finishes before $x_{(k+1)(i+1)}$. In particular, we require $(k+1)(r+1)\geq Q_j-1$ or $r+1 \geq \frac{Q_j-1}{k+1}$ and hence $r> \frac{Q_j}{2k}=r_j$. 
Observe first that $I_{v_1}$ starts before $x_1$. For $i>1$, we find using
the induction hypothesis that $I_{v_i}$ starts before $x_{(k+1)i}$, since $v_iv_{i-1}$ is an edge. We now show how to conclude from this that $I_{v_i}$ finishes before $x_{(k+1)(i+1)}$.

Since $x_{(k+1)i} \in I_{v_i}$, if we also have
$x_{(k+1)(i+1)}\in I_{v_i}$ then $I_{v_i}$ contains at least $k+1$ elements contributing
to $\mu_j$. Hence the intervals of $k+1$ new vertices of $V_j$ must start at each of
these elements. But $x_{(k+1)(i+1)}$ is contained in at most $k+1$
intervals, hence the interval of at least one of the new vertices finishes  before
$x_{(k+1)(j+1)}$. In this case, $I_{v_{i}}$  strictly contains an interval
$I_v$ with $v\in V_j$, contradicting the fact that $v_{i}$ is in $G_{\ge j}$ and $v$ is maximal in $V(G_{\ge j})$. 
\end{proof}

We are now ready to define the desired cover (actually a partition) of
$V(G)$ into bounded sets of $r$-multiplicity at most 2.
For each $j\in[p]$ and each $j$-section $S$, each  $r_j$-component of
the associated vertices of $V_i$ (where the graph $G$ is used to
define the components) is called an \emph{initial cluster} of label
$j$, or an \emph{initial $j$-cluster}. Our goal now is to merge the different initial clusters into
sets, such that the resulting family has $r$-multiplicity at most 2.
We first make some observations about the initial clusters.
\begin{lemma}\label{lem:ini_clus}
The following holds:
\begin{enumerate}
  \item Every initial $j$-cluster is $R_j$-bounded.
    \item Every $r$-ball with $r \leq r_j/2$ intersects at most one
      initial $j$-cluster in each $j$-section.
    \item In the graph $G_{\ge j}$, each $r$-ball, with $r\leq  r_j/2$, intersects at most two initial $j$-clusters. If there are two such clusters, they belong to consecutive $j$-sections.
\end{enumerate}
\end{lemma}
\begin{proof}
  Each $j$-section has weight at most $Q_j$, and thus by
  Lemma~\ref{lem:diam}, the diameter of an initial $j$-cluster is at
  most $(k+1)(Q_j+1)r_j=(k+1)(2kr_j+1)r_j\leq 3k(k+1)r_j^2 =R_j$, thus
  proving (1). (2) follows from the definition of $r_j$-components. (3) follows by (2) and Lemma~\ref{lem:non_consecutive}.
\end{proof}

We are now ready to prove Theorem~\ref{thm:pw}.

\medskip

\emph{Proof of Theorem~\ref{thm:pw}.} Let $G$ be a graph of pathwidth
at most $k$ and $r\ge 1$ be an integer. We use all the notation that
has been introduced above (note in particular that the definition of
the integers $r_j$ and $R_j$ depends on $r$). For each $j\in [p]$
(recall that $p\le k+1$), we divide $V_j$ into $j$-sections, and then
each $j$-section into initial $j$-clusters,
as explained above.
For each $j=1,\ldots,p$, we now define $j$-merged clusters
inductively as follows. The 1-merged clusters are precisely the
initial 1-clusters.

For $j\geq 2$, each $j$-merged cluster is either an initial
$j$-cluster, or the union of a $(j-1)$-merged cluster and a (possibly
empty) set of initial clusters of $P_j$. Each initial $j$-cluster will
be included in a single $j$-merged cluster, and in particular the set
of $j$-merged clusters form a partition of $V_1\cup \dots \cup V_j$. 
We will maintain the following properties by induction on $j$.
\begin{enumerate}[(i)]
    \item Each $j$-merged cluster contains a unique initial cluster $X$ of minimum label $i$. The cluster $X$ is called the \emph{oldest cluster} and $i$ is called the \emph{age} of the $j$-merged cluster.
    \item If a $j$-merged cluster contains initial clusters of labels in $[j_1,j_2]\subseteq[j]$ then the diameter of the $j$-merged cluster is at most $(1+j_2)R_{j_1}$.
    \item Every $r_j/2$-ball in $G$ intersects at most two $j$-merged clusters. 
\end{enumerate}
The merging works as follows. Let $C$ be an initial $j$-cluster, and let
$\mathcal{C}$ the set of $(j-1)$-merged clusters at distance at most
$r_j$ from $C$ in $G$. Note that $\mathcal{C}$ might be empty in which
case $C$ will become a $j$-merged cluster on its own. Otherwise,
arbitrarily choose a cluster $C'\in\mathcal{C}$ of minimal age and
merge $C$ and $C'$ (i.e. the $j$-merged cluster containing $C$ will be
the union of $C'$ and all the initial $j$-clusters that have merged
onto it).

The $p$-merged clusters are our desired partition of $V(G)$: since
$r_p>2r$, property (iii) ensures that every ball of radius at most $r$
intersects at most two $p$-merged clusters, and thus the family of
$p$-merged clusters has $r$-multiplicity at most 2. Moreover (ii)
gives each $p$-merged cluster has diameter at most
$(k+2)R_1=O((kr)^{2k+2})$, as desired.

\medskip

We now prove that properties (i)--(iii) hold by induction on $j\ge 1$.
For $i=1$, the $1$-merged clusters are the initial 1-clusters, which
are $R_1$-bounded by Lemma~\ref{lem:ini_clus}(1). This proves (i) and
(ii). Since $G=G_{\ge 1}$, (iii) is a direct consequence of Lemma~\ref{lem:ini_clus}(3).
We now assume we have proved the properties for $j-1\geq 1$  and prove the assertions for $j$.

\medskip
\noindent \textit{(i) Each $j$-merged cluster contains a unique initial cluster $X$ of minimum age.}\\
This is true just before the merging for $j-1$ by induction. Note that
we
merge each initial $j$-clusters with a single $(j-1)$-merged cluster or leave the initial $j$-cluster alone. So the conclusion indeed follows.
\smallskip

\noindent
\textit{(ii) If a $j$-merged cluster contains initial clusters of labels in $[j_1,j_2]$ then the diameter of the merged cluster is at most $(1+j_2)R_{j_1}$.}\\ 
For any $j$-merged cluster that just consists of an initial $j$-cluster, the
claim follows immediately from Lemma~\ref{lem:ini_clus}(1). If a
$j$-merged cluster is also a $(j-1)$-merged cluster (i.e.\ no
initial $j$-cluster was added to it), then the property follows by
induction. Finally, consider a $j$-merged cluster $C$ obtained from a
$(j-1)$-merged cluster $C'$ of age $j_1$ by adding some initial
$j$-clusters to it. Then $C$ also has age $j_1$
and it only contains initial clusters whose labels lie in
$[j_1,j]$. All the initial $j$-clusters contained in $C$ have diameter
at most $R_j$ and are at distance at most $r_j$ from $C'$ (whose
diameter is at most  $(1+j-1)R_{j_1} $ by induction). It follows that
the diameter of $C$ is at most $(1+j-1)R_{j_1}+2r_j+2R_j \le (1+j)
R_{j_1}$, since $2r_j+2R_j \le R_{j_1} $, as desired.

\smallskip

\noindent \textit{(iii) Every $r_j/2$-ball in $G$ is incident to at most two $j$-merged clusters.} \\
Fix some arbitrary vertex $v\in G$, and consider the ball
$B(v,r_{j}/2)$. Note that it follows from Lemma \ref{lem:ini_clus}(3)
that it intersects at most two $j$-merged clusters if the ball is
included in $V(G_{\ge j})$, so we can assume that $B(v,r_{j}/2)$
intersects some $(j-1)$-merged cluster, and thus some $j$-merged
cluster $C$ of age $i<j$.

We now claim that the age of all the $j$-merged clusters intersecting $B(v,r_j/2)$ has to be smaller than $j$. Indeed, all initial $j$-clusters at distance at most $r_j/2$ of $v$ are at distance at most $r_j$ from $C$. Hence any such initial cluster will be merged onto a cluster of age at most the age of $C$. 

Let $A,B,C$ be three $j$-merged clusters intersecting
$B(v,r_j/2)$. Be the previous paragraph, the age of each of $A,B,C$ is less than $j$. Hence there
exist $(j-1)$-merged clusters $A',B',C'$ that have been transformed
into the $j$-merged clusters $A,B,C$ (possibly by adding an empty set
of initial $j$-clusters). 
Since the initial $j$-clusters have diameter at most $R_j$, and are
within $r_j$ of the $(j-1)$-merged clusters they merge onto (if any), the vertex $v$
is at distance at most $r_j/2+R_j+r_j < r_{j-1}/2$ from $A'$, $B'$ and $C'$.
This contradicts the induction hypothesis, and concludes the proof of
Theorem~\ref{thm:pw}.\hfill $\Box$

\bigskip

The \emph{layered pathwidth} of a graph can be defined analogously as the layered
treewidth (see Section~\ref{sec:topo}), by restricting the tree-decomposition to be a path. Using
the same proof as that of Theorem~\ref{thm:ltwasdim}, and replacing
(layered) treewidth by (layered) patwidth, we obtain the following immediate
consequence of Theorem~\ref{thm:pw}.

\begin{corollary}\label{cor:lpwasdim}
  For any $k$, the class of  graphs of layered pathwidth at most $k$
  has asymptotic dimension at most 2.
\end{corollary}

A graph $H$ is an \emph{apex-forest} if $H-v$ is a forest for some
$v\in V(H)$. It was proved in~\cite{DEJMW} that for any apex-forest
$H$, the class of $H$-minor free graphs has bounded layered
pathwidth. We obtain the following as a direct consequence.

\begin{corollary}\label{cor:apexforest}
  For any apex-forest $H$, the class of $H$-minor free graphs
  has asymptotic dimension at most 2.
\end{corollary}

\section{Conclusion and open problems}\label{sec:ccl}

Recall that Question~\ref{qn:fp}, by Fujiwara and Papasoglu~\cite{FP20} asks whether there is an constant
$k$ such that any proper minor-closed class has asymptotic dimension
at most $k$ (and speculating that this might even be true with
$k=2$).
While we have proved that this is true for classes of graphs excluding
a $K_{3,p}$ minor (Theorem~\ref{thm:k3p}), or an apex-forest minor
(Corollary~\ref{cor:apexforest}) and for classes of graphs of bounded degree
excluding a fixed minor (Theorem~\ref{thm:minordeg}), we do not know how to answer this question in
general. In fact we do not know the answer to the following more
specific question.

\begin{qn}\label{qn:tw}
Is there a constant $k$ such that for any $t\ge 0$, the class of
graphs of treewidth at most $t$ has asymptotic dimension at most
$k$?
\end{qn}

As we have proved in Theorem~\ref{thm:ltwasdim}, if this question had
a positive answer (for some integer $k$), then any class of graphs of bounded layered
treewidth would have asymptotic dimension at most $k+1$, and in particular
for any apex graph $H$, the class of $H$-minor free graphs would have
asymptotic dimension at most $k+1$. 

Answering Question~\ref{qn:tw} positively might also lead to progress on Question
~\ref{qn:fp}.  The approach here would be to use
the Graph Minor Structure Theorem of Robertson and Seymour~\cite{RS03}, which roughly states that any proper minor-closed class can be obtained from graphs that are almost embeddable on a fixed surface, by gluing them along cliques in a tree-like way. Using Theorem~\ref{thm:ltwasdim}, a positive answer to Question~\ref{qn:tw} would imply that graphs that are almost embeddable on a fixed surface have bounded asymptotic dimension, and we hope that it would also provide a systematic way to combine bounded covers of bounded multiplicity along some tree-decomposition, giving a potential approach to answer Question
~\ref{qn:fp} positively.

By Corollary~\ref{cor:gkplanar}, a negative answer to the following question
 would imply a negative answer to Question~\ref{qn:tw}. 

\begin{qn}\label{qn:1planar}
Do 1-planar graphs have bounded asymptotic dimension?
\end{qn}

Recall that Ostrovskii and Rosenthal~\cite{OR15} proved that the class
of 
$K_t$-minor free graphs has asymptotic dimension at most
$4^t$. Independently of the answer of Question~\ref{qn:fp}, improving the bound (as a function of $t$) is
an interesting open problem. The following possible extension of
Theorem~\ref{thm:k3p} is fairly natural.

\begin{qn}\label{qn:kst}
Is it true that for every integers $s$ and $t$, the class of graphs with no $K_{s,t}$ minor has
asymptotic dimension at most $s-1$? Or more generally at most $f(s)$, for some function $f$?
\end{qn}

\medskip

It should be noted that Questions~\ref{qn:fp} and~\ref{qn:tw} are
already open if we fix the radius $r=1$. As explained in
Section~\ref{sec:clustered}, in this case the questions become:

\begin{qn}\label{qn:fpr=1}
Is there a constant $k$ such that for any $t\ge 0$, the class of
graphs with no $K_t$ minor has weak diameter chromatic number at most
$k$? Can we take $k=3$?
\end{qn}

\begin{qn}\label{qn:twr=1}
Is there a constant $k$ such that for any $t\ge 0$, the class of
graphs of treewidth at most $t$ has weak diameter chromatic number at
most $k$? Can we take $k=2$?
\end{qn}

As observed in~\cite[Theorem 4.1]{LO18}, Question~\ref{qn:twr=1} has a
negative answer if we replace the weak diameter by the strong diameter
(i.e.\ if we ask that the subgraph induced by each monochromatic
component has
bounded diameter, at least $t+1$ colours are necessary in some graphs
of treewidth at most $t$).

\medskip

For some function $f$, we say that a class of graphs $\mathcal{G}$ has
\emph{expansion} at most $f$ if any minor obtained from contracting connected
subgraphs of radius at most $r$ in a graph of $\mathcal{G}$ has
average degree at most $f(r)$ (see~\cite{NO12} for more details on this
notion). 
In Section~\ref{sec:geom} we have proved that classes of graphs of
polynomial growth have bounded asymptotic dimension. Note that if a class has bounded (resp.\ polynomial)
growth, then it has bounded (resp.\ polynomial)
expansion.

\begin{qn}\label{qn:polyexp}
Is it true that every class of graphs of polynomial expansion has
bounded asymptotic dimension? 
\end{qn}

Observe that polynomial expansion would again be best possible here, as we
have constructed classes of graphs of (barely) superpolynomial growth
(and therefore expansion) with unbounded asymptotic dimension.
It should be noted that there are important connections between
polynomial expansion and the existence of strictly sublinear
separators~\cite{DN16}. On the other hand, Hume~\cite{Hum17} proved
that classes of graphs of bounded growth and bounded asymptotic dimension have
sublinear separators.

\medskip

We conclude with a remark and a question on $q$-fat minors, which where introduced in Section~\ref{sec:fatminor}. We note that the proof of Ostrovskii and Rosenthal~\cite{OR15} can be adapted to prove that for any integers $q$ and $t\ge 3$, the class of graphs with no $q$-fat $K_t$-minor has asymptotic dimension at most $4^t$. We believe that excluding a $K_t$-minor or a $q$-fat $K_t$-minor should not make a difference for the asymptotic dimension, and that $q$-fat minors are the key notion to answer Question~\ref{qn:fp} and~\ref{qn:kst} (in the same way $q$-fat bananas were the key to the proof of Theorem~\ref{thm:k3p}).

\begin{qn}\label{qn:qfat}
Is it true that for any integers $q$ and $t\ge 3$, the class of $K_t$-minor free graphs and the class of $q$-fat $K_t$-minor free graphs have the same asymptotic dimension?
\end{qn}

\noindent \textbf{Note added.} After we made our manuscript public, Liu~\cite{Liu20} gave a positive answer to Question~\ref{qn:fp}, in turn implying positive answers to Questions~\ref{qn:tw}, \ref{qn:1planar}, \ref{qn:kst}, \ref{qn:fpr=1}, \ref{qn:twr=1}, and \ref{qn:qfat}.

\begin{acknowledgement}
The authors would like to thank Arnaud de Mesmay, Alfredo Hubard and
Emil Saucan
for the discussions on the discretization of Riemannian surfaces, and
Ga\"el Meigniez for suggesting the simple proof of
Lemma~\ref{lem:rietogr} on mathoverflow. The authors would also like to thank David Hume for interesting discussions.
\end{acknowledgement}

\appendix

\section{Proof of Theorem~\ref{thm:mthm49}}

In this appendix we  prove Theorem~\ref{thm:mthm49}. We closely
follow~\cite{BDLM}.

\subsection{Control function} For two integers $n\ge 0$ and $k\geq n+1$ and a metric space $X$, we say that a
function $D_X$ is an $(n,k)$-dimensional control function for $X$ if
for all $r>0$, there are subsets $U_1,\dots,U_k$ of $X$ such that 
\begin{enumerate}
    \item Each $r$-component of $U_i$ is $D_X(r)$-bounded.
    \item Each $x\in X$ belongs to at least $k-n$ sets $U_i$.
    \end{enumerate}

Note that an $(n,n+1)$-dimensional
control function is the same as an $n$-dimensional control function.
    We will need the following result.

\begin{lemma}[Theorem 2.4 in~\cite{BDLM}]\label{lem:thm24}
If $X$ admits an $(n,k)$-dimensional control function $D$,
then $D'(r):=D(3r)+2r$ is an $(n,k+1)$-dimensional control function for
$G$.
\end{lemma}

Note that $\mathbb{R}$ has Assouad-Nagata dimension 1, so in
particular Lemma~\ref{lem:thm24} implies that it has a linear
$(1,n)$-control function for any $n\ge 2$. Observe that such a function (and
the associated sets $U_1,\ldots,U_n$) can be
constructed  explicitly in this case (and with better bounds than those
implied by repeated applications of Lemma~\ref{lem:thm24}). We first divide
$\mathbb{R}$ into consecutive intervals $V_k:=[rk,r(k+1))$, $k\in
\mathbb{Z}$ of length $r$. For any $1\le i \le n$, we then define
$$U_i=\mathbb{R}\setminus \bigcup \{V_k\,|\, k\in \mathbb{Z}, k\equiv i\bmod n\}.$$ Note each $r$-component of $U_i$ is a
single interval of length $(n-1)r$, and is thus $(n-1)r$-bounded. We write this as an observation for future
reference.

\begin{obs}\label{obs:cfR}
  For any $n\ge 1$, and $r>0$, there are subsets $(U_i)_{1\le i \le n}$ of
  $\mathbb{R}$, such that for each $1\le i\le n$, each $r$-component
  of $U_i$ is $(n-1)r$-bounded, and each element of $\mathbb{R}$ is
  contained in at least $n-1$ sets.
\end{obs}

\subsection{Control and dimension of real projections}

Let $X$ be a metric space and $f :X\to \mathbb{R}$ be a real
projection of $X$.
For $n\ge 0$ and $k\geq n+1$, a function $D_f$ is an \emph{$(n,k)$-dimensional
control function} for $f$ if for all $r,S>0$, and any
$(\infty,S)$-bounded set $A$ (with respect to $f$), there are subsets $A_1,\dots,A_k$ of
$A$ such that
\begin{enumerate}
    \item all $r$-components of $A_i$ are $D_f(r,S)$-bounded
    \item any $x\in A$ belongs to at least $k-n$ elements of $\{A_1,\ldots,A_k\}$.
    \end{enumerate}

 As before, an $(n,n+1)$-dimensional
control function for $f$ is the same as an $n$-dimensional control
function for $f$, defined in Section~\ref{sec:controlfun}. Recall that
a control function $D_f$ for a real projection $f$ is \emph{linear}
if $D_f (r,S)= ar+bS+c$, for some constants
$a,b,c>0$, and a \emph{dilation} if $D_f (r,S)= ar+bS$, for some constants
$a,b>0$.

\medskip

Control functions for real projections satisfy the following counterpart of Lemma~\ref{lem:thm24}.

\begin{lemma}[Proposition 4.7 in~\cite{BDLM}]\label{lem:prop47}
Let $X$ be a metric space and $f:X\to \mathbb{R}$ be a real projection
of $X$. If  $f$ admits an $(n,k)$-dimensional control function $D_f$,
then $D_f'(r,S):=D_f(3r,S)+2r$ is an $(n,k+1)$-dimensional control function
for $f$. In particular, if $D_f$ is linear, then $D_f'$ is also
linear, and $D_f$ is dilation, then $D_f'$ is also a dilation.
\end{lemma}

The definition of control functions for a real projection $f:X\to
\mathbb{R}$ of $X$ involves
$(\infty,S)$-bounded sets, which are the preimage of $S$-bounded sets
of $\mathbb{R}$ under $f$. The following
stronger version can be easily derived from the definition.

\begin{lemma}[Proposition 4.8 in~\cite{BDLM}]\label{lem:prop48}
Let $X$ be a metric space and $f:X\to \mathbb{R}$ be a real projection
of $X$. If  $f$ admits an $(n,k)$-dimensional control function $D_f$, then for any $r,s,S>0$, and any subset
$B\subseteq \mathbb{R}$ whose $s$-components are $S$-bounded, the preimage
$f^{-1}(B)$ of $B$ in $X$ can be covered by $k$ sets
with $(r,s)$-components that are  $D_f(r,S)$-bounded, and such that any
element of $f^{-1}(B)$ is covered by at least $k-n$ sets.
\end{lemma}

\subsection{Union of sets}

We will need the following result on the union of sets.

\begin{lemma}[Lemma 3.7 in~\cite{BDLM}]\label{lem:lem37}
Let $X$ be a metric space and $f:X\to \mathbb{R}$ be a real projection
of $X$. Let $A_1$ and $A_2$ be subsets
of $X$ such that all $(r_i,s_i)$-components of $A_i$ are
$(R_i,S_i)$-bounded, for $i\in\{1,2\}$. If $R_1+2r_1<r_2$ and
$S_1+2s_1<s_2$ then all $(r_1,s_1)$-components of $A_1\cup A_2$ are
$(R_2+2r_2,S_2+2s_2)$-bounded.
\end{lemma}

\subsection{Kolmogorov trick for real projections}

We are now ready to prove Theorem~\ref{thm:mthm49}. We divide it into
two separate parts, for convenience. The first part is proved explicitly
in~\cite{BDLM} (so we do not reprove it here).

\begin{thm}[Theorem 4.9 in~\cite{BDLM}]\label{thm:mthm49b}
  Let $X$ be a metric space and $f:X\to \mathbb{R}$ be a real
  projection of $X$.
If $f$ admits an $n$-dimensional
control function $D_f$, then $X$ admits an $(n+1)$-dimensional
control function $D_X$ such that $D_X(r)$ only depends on $r$, $D_f$ and
$n$. Moreover, if $D_f$ is linear then $D_X$ is also linear and if $D_f$ is
a dilation then $D_X$ is also a dilation.
\end{thm}

The second part is not stated explicitly in~\cite{BDLM} but follows
from an optimized version of the proof of
their Theorem 4.9. We prove it here for completeness.

\begin{thm}[Theorem 4.9 in~\cite{BDLM}]\label{thm:mthm49c}
  Let $X$ be a metric space and $f:X\to \mathbb{R}$ be a real
  projection of $X$. Let $a,b\ge 1$ be real numbers.
If $f$ admits an $n$-dimensional
control function $D_f$ with $D_f(r,S)\le ar+bS$, for any $r,S>0$, then  $X$ admits an $(n+1)$-dimensional
control function $D_X$ such that $D_X(r)\le 20a(6a+b(n+4))\cdot r$.
\end{thm}

\begin{proof}
Given $r>0$, we set $s_1:=r$; $S_1:=(n+1)s_1$; and
$s_2:=S_1+2s_1+r=(n+4)r$ (note that by definition, $S_1+2s_1<s_2$).

For $i\in\{1,2\}$, we consider the set $B_i=\bigcup_{k\in
  \mathbb{Z}}\big [ (2k+i-1)s_2,(2k+i)s_2\big )$. Note that the sets $B_1$ and $B_2$
partition $\mathbb{R}$ and any $s_2$-component of $B_i$
($i\in\{1,2\}$) is $s_2$-bounded.

By Observation~\ref{obs:cfR}, there
are subsets $(B_1^j)_{1\le j \le n+2}$ of $B_1$, such that for each $1\le j\le n+2$, each $s_1$-component
  of $B_1^j$ is $(n+1)s_1$-bounded (and thus $S_1$-bounded), and each element of $B_1$ is
  contained in at least $n+1$ sets $B_1^j$.

  \smallskip

Recall that an $n$-dimensional control function for $f$ is the same as an
$(n,n+1)$-control function for $f$. By Lemma~\ref{lem:prop47}, the
function $D_f'(r,S):=D_f(3r,S)+2r$ is an $(n,n+2)$-dimensional control function
for $f$.
We now define
$r_1:=r$; $R_1:=D_f'(r_1,s_2)$; $r_2:=R_1+2r_1+r$ (observe that $R_1+2r_1<r_2$); and
$R_2:=D_f'(r_2,s_2)$. By Lemma~\ref{lem:prop48}, for each $i\in \{1,2\}$, since the
$s_2$-components of $B_i$ are $s_2$-bounded,
the subset $f^{-1}(B_i)$ of elements of $X$ can be covered by $n+2$ sets $(A_i^j)_{1\le j
\le n+2}$
with $(r_i,s_2)$-components that are  $D_f'(r_i,s_2)$-bounded (and thus $R_i$-bounded), and such that any
element in $f^{-1}(B_i)$ is covered by at least $n+2-n=2$ sets $A_i^j$.
(In what follows the property that the elements are covered at least twice instead of once will only be used for $i=1$.)

\smallskip

For any $1\le j \le n+2$, set $D^j:=(A_1^j \cap
f^{-1}(B_1^j))  \cup A_2^j$. We
first observe that the sets $D^j$ form a cover of $X$.
To see
this, consider an element $x\in X$. Note
that $f(x)\in B_i$ for some $i\in \{1,2\}$ since $B_1$ and $B_2$
partition $\mathbb{R}$. If $f(x)\in B_2$, then $x$ is covered by a
set $A_2^j\subseteq D^j$, for some $1\le j \le n+2$. Assume now that $x\in B_1$. Then $f(x)$ is
covered by at least $n+1$ sets $B_1^j$, while $x$ (as an element of $f^{-1}(B_1)$) is also covered by at least two sets $A_1^j$. Since $1\le j
\le n+2$, there is an index $j$ such that $v\in A_1^j \cap
f^{-1}(B_i^j)\subseteq D^j$, as desired.

Fix some $1\le j \le n+2$ and some $i\in \{1,2\}$. Recall that all $(r_i,s_2)$-components of $A_i^j$
are $(R_i,s_2)$-bounded. Consider two elements 
$x,y$ lying in the same $(r_1,s_1)$-component of $A_1^j \cap
f^{-1}(B_1^j)$. Since $s_1\le s_2$, $x$ and $y$ are included in a
$(r_1,s_2)$-component of $A_i^j$ and thus $\{x,y\}$ is
$(R_1,s_2)$-bounded. On the other hand, since all
$s_1$-components of $B_1^j$ are $S_1$-bounded, $|f(x)-f(y)|\le S_1$, and thus $\{x,y\}$ is
$(R_1,S_1)$-bounded.
This proves that all
$(r_1,s_1)$-components of $A_1^j \cap
f^{-1}(B_1^j)$ are $(R_1,S_1)$-bounded.

Recall that all $(r_2,s_2)$-components of $A_2^j$
are $(R_2,s_2)$-bounded.
By Lemma~\ref{lem:lem37}, since $R_1+2r_1<r_2$ and
$S_1+2s_1<s_2$, this implies that all $(r_1,s_1)$-components of $D^j=(A_1^j \cap
f^{-1}(B_1^j)  )\cup A_2^j$ are
$(R_2+2r_2,s_2+2s_2)$-bounded. Since
$(r_1,s_1)=(r,r)$, $(r_1,s_1)$-components of $D^j$ are the same as
$r$-components of $D^j$, so it follows that $r$-components of $D^j$
are $(R_2+2r_2)$-bounded. We obtain that $D_X(r):=R_2+2r_2$ is an $(n+1)$-dimensional
control function for $X$.

Recall that $D_f(r,S)\le ar+bS$, for some constants $a\ge 1,b\ge 1$. It
follows that $D_f'(r,S)= D_f(3r,S)+2r \le (3a+2)r+bS$. As a consequence, $R_1\le
(3a+2)r+bs_2\le (3a+2+b(n+4))\cdot r$ and $r_2= R_1+3r\le
(3a+5+b(n+4))\cdot r$. This implies $R_2\le ((3a+2)
(3a+5+b(n+4))+b(n+4))\cdot r\le ( 40a^2+6ab(n+4) )\cdot r$, using that $a,b\ge
1$. We obtain $$D_X(r)=R_2+2r_2\le 20a(6a+b(n+4))\cdot r,$$ as desired.
\end{proof}

\end{document}